\documentclass[11pt, reqno]{amsart}
\usepackage[section]{uniqueness}

\newcommand*{\const}{\varrho}
\newcommand*{\var}{\sigma}

\title[The steady states of strong-KPP reactions in general domains]{The steady states of strong-KPP reactions\\ in general domains}

\author{Henri Berestycki}
\address{HB: Centre d'analyse et de math\'{e}matique sociales, EHESS-CNRS, 54 Boulevard Raspail, 75006, Paris, France}
\address{\hspace{1.8em}Senior Visiting Fellow, Institute for Advanced Study, Hong Kong University of Science and Technology}
\email{\tt hb@ehess.fr}

\author{Cole Graham}
\address{CG: Division of Applied Mathematics, Brown University, 182 George St, Providence, RI 02906, USA}
\email{\tt cole\_graham@brown.edu}

\begin{document}
\begin{abstract}
  We study the uniqueness of steady states of strong-KPP reaction--diffusion equations in general domains under various boundary conditions.
  We show that positive bounded steady states are unique provided the domain satisfies a certain spectral nondegeneracy condition.
  We also formulate a number of open problems and conjectures.
\end{abstract}

\maketitle

\section{Overview and main results}

We study the steady states of KPP-type reaction--diffusion equations in general domains in $\R^d$.
In particular, we consider the existence and uniqueness of positive bounded solutions to semilinear elliptic equations of the form
\begin{equation}
  \label{eq:main-Dirichlet}
  \begin{cases}
    -\Delta u = f(u) & \text{in } \Omega,\\
    u = 0 & \text{on } \partial \Omega
  \end{cases}
\end{equation}
in nonempty connected open domains $\Omega \subset \R^d$, which need not be bounded.
We also consider several other types of boundary conditions.
We state our assumptions on $f$ and $\Omega$ below.

\subsection*{Motivation.}
Reaction--diffusion equations describe a host of phenomena in diverse fields including chemistry, ecology, and sociology.
For instance, a solution $u$ of \eqref{eq:main-Dirichlet} could model the population density of a species in a geographic region $\Omega$.
The reaction $f$ represents the growth rate of the species and the Dirichlet boundary models a hostile border that kills individuals.
A solutions of \eqref{eq:main-Dirichlet} represents an equilibrium or steady state between these two competing factors.
The study of this system naturally begins with a fundamental question: are such steady states unique?
This question has proved to be exceedingly rich.

The steady states of many homogeneous reaction--diffusion equations posed on the whole space $\R^d$ are well understood; for more details, see the discussion of related works below.
Much less is known on other domains.
Here, we examine the influence of geometry on the solutions of \eqref{eq:main-Dirichlet}.
We study homogeneous equations in subsets of $\R^d$ with Dirichlet, Robin, and Neumann boundary conditions.
As we shall see, absorbing boundary conditions of Dirichlet or Robin type exhibit the most complex behavior.

We are further motivated by parabolic reaction--diffusion equations, which represent population \emph{dynamics}.
The parabolic form of \eqref{eq:main-Dirichlet} describes the propagation of the species throughout its environment, and an elliptic solution represents a fixed point of these dynamics.
Typically, stable steady states are attractors of the parabolic equation at long times, and unstable steady states separate domains of attraction.
The classification of solutions to the elliptic problem \eqref{eq:main-Dirichlet} is thus an essential element of a broader understanding of the time-dependent dynamics.

In our context, several prior works have treated bounded domains; these serve as important precursors for this paper.
However, we expect the parabolic dynamics to be of greater interest in \emph{unbounded} domains.
On unbounded domains, solutions have sufficient space to develop coherent long-time patterns like asymptotic spreading speeds and traveling waves.
We are therefore keenly interested in the classification of steady states in the unbounded case.
A handful of prior works have considered \emph{specific} unbounded domains including epigraphs, half-spaces, and cones.
However, to our knowledge, the nature of steady states on general unbounded domains is largely open.

We aim to fill this gap.
In this paper, we show that nontrivial steady states of so-called strong-KPP reactions are unique under quite general conditions on the domain and boundary data.
Moreover, we demonstrate that the strong-KPP hypothesis cannot be easily weakened: a slightly more general class of reactions can support multiple nontrivial steady states.
In a forthcoming paper, we take up the study of more general reactions and obtain various conditions on the domain that yield uniqueness.
In view of our previous work on the half-space \cite{BG}, the present paper is the second in a sequence focused on the steady states and long-time dynamics of reaction--diffusion equations in the presence of boundary.

\subsection*{Setup.}
We now describe our objects of study in greater detail.
To begin, we state our hypotheses on the reaction $f$ and the domain $\Omega$.

We assume that the nonlinearity ${f \colon [0, \infty) \to \R}$ is $\m{C}^{1,\gamma}$ for some $\gamma \in (0, 1]$ and satisfies $f(0) = f(1) = 0$.
As a consequence, $0$ solves \eqref{eq:main-Dirichlet} and $1$ is a supersolution, in the sense that it satisfies \eqref{eq:main-Dirichlet} with $\geq$ in place of equalities.
We also assume that $f|_{(1, \infty)} < 0$, so that the reaction drives large values down towards $1$.
The maximum principle then implies that all positive bounded solutions of \eqref{eq:main-Dirichlet} take values between $0$ and $1$.
We are thus primarily interested in the behavior of $f$ on $(0, 1)$.
We consider three nested classes of reactions: \emph{positive}, \emph{weak-KPP}, and \emph{strong-KPP}.

All our reactions will be \emph{positive} in the sense that:
\begin{enumerate}[label = \textup{(P)}, leftmargin = 5em, labelsep = 1em, itemsep= 1ex, topsep = 1ex]
\item
  \label{hyp:positive}
  $f|_{(0,1)} > 0$ and $f'(0^+) > 0$.
\end{enumerate}
Positive reactions are sometimes known as ``monostable.''
We use distinct terminology to emphasize the positive derivative at $u = 0$, which plays a significant role in our analysis.
In many sources, ``monostability'' only denotes the first condition in \ref{hyp:positive}.

Our \emph{weak-KPP} reactions are positive and additionally satisfy:
\begin{enumerate}[label = \textup{(W)}, leftmargin = 5em, labelsep = 1em, itemsep= 1ex, topsep = 1ex]
\item
  \label{hyp:weak}
  $f(s) \leq f'(0)s$ for all $s \in [0, 1]$.
\end{enumerate}
Finally, \emph{strong-KPP} reactions are weak-KPP and satisfy a form of strict convexity:
\begin{enumerate}[label = \textup{(S)}, leftmargin = 5em, labelsep = 1em, itemsep= 1ex, topsep = 1ex]
\item
  \label{hyp:strong}
  The map $s \mapsto \frac{f(s)}{s}$ is strictly decreasing on $(0, 1]$.
\end{enumerate}
Both \ref{hyp:weak} and \ref{hyp:strong} have been termed the ``KPP condition'' in recognition of the pioneering work of Kolmogorov, Piskunov, and Petrovsky~\cite{KPP} on a similar class of reactions.
In many settings, the distinction is unimportant.
For example, weak and strong-KPP reactions exhibit identical propagation phenomena in the whole space.
The same holds in half-spaces~\cite{BG} and cones~\cite{LL}.

However, the steady states of weak and strong-KPP reactions behave quite differently in general domains.
As we shall see, the strong KPP hypothesis \ref{hyp:strong} ensures broad uniqueness for steady states.
In marked contrast, both positive and weak-KPP reactions can easily support multiple steady states.
These divergent behaviors lead us to carefully delineate between the hypotheses \ref{hyp:weak} and \ref{hyp:strong}.
We devote most of our attention to strong-KPP reactions.

Throughout, we assume that the domain $\Omega$ is open, nonempty, connected, and uniformly $\m{C}^{2, \gamma}$ smooth.
We make this precise in Definition~\ref{def:smooth} in the appendix.
For now, we merely note that this hypothesis includes a uniform interior ball condition.

In the remainder of the paper, we consider more general boundary conditions.
Let $\nu$ denote the outward normal vector field on $\partial \Omega$.
Given $\const \in [0, 1]$, let $\m{N}_\const \colon \m{C}^1\big(\bar\Omega\big) \to \m{C}^0(\partial \Omega)$ denote the boundary operator
\begin{equation}
  \label{eq:boundary-op}
  \m{N}_\const u \coloneqq \const \partial_\nu u + (1 - \const) u|_{\partial \Omega}.
\end{equation}
Then we study a natural generalization of \eqref{eq:main}:
\begin{equation}
  \label{eq:main}
  \begin{cases}
    -\Delta u = f(u) & \text{in } \Omega,\\
    \m{N}_\const u = 0 & \text{on } \partial \Omega.
  \end{cases}
\end{equation}
The values $\const = 0$, $\const \in (0, 1)$, and $\const = 1$ correspond to Dirichlet, Robin, and Neumann boundary conditions, respectively.

Even in its weak form, the KPP condition \ref{hyp:weak} ensures that the ``rate of growth'' $f(s)/s$ is maximized at $s = 0$, where $f$ resembles its linearization $f'(0)s$.
KPP phenomena are thus often ``linearly determined:'' they can be understood via the linearization of \eqref{eq:main} about $0$.
And indeed, much of our analysis revolves around spectral properties of the linearization.
We express our results in terms of the generalized principal eigenvalue introduced in \cite{BNV} and applied to unbounded domains in \cite{BR}.

Let $\m{L}$ be an elliptic operator on $\Omega$.
Then the \emph{generalized principal eigenvalue} of $-\m{L}$ on $\Omega$ with boundary parameter $\const$ is given by
\begin{equation}
  \label{eq:lambda}
  \lambda(-\m{L}, \Omega, \const) \coloneqq \sup\big\{\mu \in \R \mid \exists \, \phi \in W_{\mathrm{loc}}^{2,d}(\Omega)_+ \hspace{1ex}\text{s.t.} \hspace{1ex} \m{N}_\const \phi \geq 0, \hspace{1ex} -\m{L} \phi \geq \mu \phi\big\}.
\end{equation}
Here $W_{\mathrm{loc}}^{2,d}(\Omega)_+$ denotes the set of positive functions in $W_{\mathrm{loc}}^{2,d}(\Omega)$.
This formulation with boundary conditions first appeared in \cite{PS}.

The definition \eqref{eq:lambda} is closely related to the principal eigenvalue familiar from bounded domains; see Proposition~\ref{prop:lambda} below for details.
We will work almost exclusively with the Laplacian, so we make liberal use of the shorthand
\begin{equation*}
  \lambda(\Omega, \const) \coloneqq \lambda(-\Delta, \Omega, \const).
\end{equation*}
As we shall see, this eigenvalue is intimately linked with existence in \eqref{eq:main}.

First, however, we take up the question of uniqueness.
In the strong-KPP case, Rabinowitz~\cite{Rabinowitz} showed that \eqref{eq:main} admits at most one positive solution on bounded domains; the first author subsequently established the same in a more general framework~\cite{Berestycki}.
Thus any obstruction to uniqueness can only occur ``at infinity.''
We therefore take some care with the structure of $\Omega$ at infinity.
Given a sequence of points $(x_n)_{n \in \N}$ in $\Omega$, we say that $\Omega^*$ is the \emph{connected limit} of $\Omega$ along $(x_n)_n$ if the sequence of domains $(\Omega - x_n)_{n \in \N}$ converges in a suitable sense and $\Omega^*$ is the connected component of the limit that contains the origin.
For a precise definition, see Definitions~\ref{def:limit} and \ref{def:connected-limit}.
This terminology accounts for the fact that connected domains may have disconnected limits; see Figure~\ref{fig:connected-limit}.
\begin{figure}[t]
  \centering
  \includegraphics[width=\linewidth]{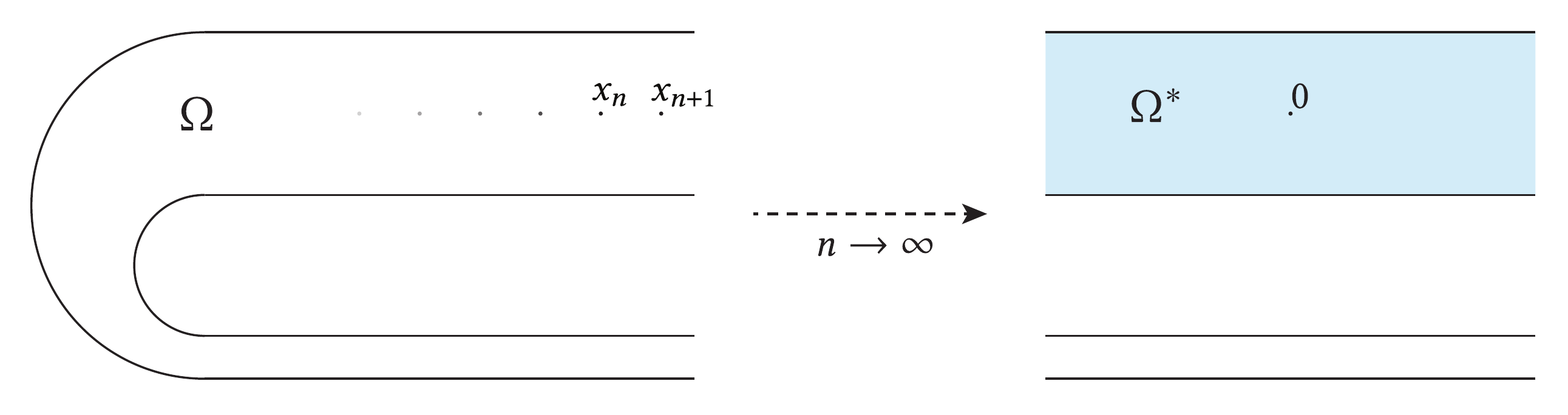}
  \caption{A horseshoe domain $\Omega$ such that $\lim_{n \to \infty}(\Omega - x_n)$ is disconnected.
    The sequence illustrated here selects the upper strip as the connected limit $\Omega^*$.
    There are exactly two other possibilities, up to translation: the lower strip and $\Omega$ itself, which is the connected limit of all bounded sequences.
  }
  \label{fig:connected-limit}
\end{figure}

We next introduce the collection of eigenvalues on all connected limits.
\begin{definition}
  The \emph{principal limit spectrum} is
  \begin{equation*}
    \Sigma(\Omega, \const) \coloneqq \big\{\lambda(\Omega^*, \const) \mid \Omega^* \text{ is a connected limit of } \Omega\big\}.
  \end{equation*}
  We refer to the elements of $\Sigma$ as \emph{(principal) limit eigenvalues}.
  We denote the closure by $\bar{\Sigma}(\Omega, \rho)$.
\end{definition}
\subsection*{Results.}
With these notions, we can state our main results for strong-KPP reactions.
We show that strong-KPP steady states are unique provided the principal limit spectrum avoids the critical eigenvalue $f'(0)$.
\begin{theorem}
  \label{thm:uniqueness}
  Suppose $f$ satisfies \ref{hyp:strong} and $f'(0) \not \in \bar{\Sigma}(\Omega, \const)$.
  Then \eqref{eq:main} has at most one positive bounded solution.
\end{theorem}
In bounded domains, this follows from Corollary 4.1 in~\cite{Rabinowitz}.
The argument in~\cite{Rabinowitz} extends to unbounded domains provided $u$ does not vanish locally uniformly along a subsequence in $\Omega$.
However, if $\Omega$ has a limit eigenvalue above $f'(0)$, then every solution \emph{does} vanish along the corresponding subsequence.
On the other hand, we show that \eqref{eq:main} satisfies the maximum principle when $\lambda(\Omega, \const) > f'(0)$.
We can combine these tools to recover uniqueness whenever the limit eigenvalues above $f'(0)$ can be cleanly separated from those below.
To this end, we prove the following geometric result that may be of independent interest.
We state the simpler Dirichlet case here; for the full form, see Theorem~\ref{thm:domain-decomposition} below.
\begin{theorem}
  \label{thm:decomp}
  A domain $\Omega$ satisfies $f'(0) \not \in \bar{\Sigma}(\Omega, 0)$ if and only if $\Omega$ can be written as the union of two uniformly $\m{C}^{2,\gamma}$ open sets $\Omega_-$ and $\Omega_+$ such that $\Sigma(\Omega_-, 0) \subset [0, f'(0))$ and $\Sigma(\Omega_+, 0) \subset (f'(0), \infty)$.
\end{theorem}
\noindent
For a visualization of the decomposition in Theorem~\ref{thm:decomp}, see Figure~\ref{fig:ample-narrow}.
The sets $\Omega_-$ and $\Omega_+$ typically overlap; thus the decomposition is by no means unique.
Moreover, the sets need not be connected and may be empty.

Our construction of $\Omega_+$ makes essential use of a deep result of Lieb concerning principal eigenvalues on intersections of open sets \cite{Lieb}.
The decomposition in Theorem~\ref{thm:decomp} is also similar in spirit to ideas at the heart of the moving plane method.
Indeed, in that method, as presented in~\cite{BN}, one decomposes $\Omega$ into two parts, one with certain compactness properties and the other with a maximum principle.
It seems possible that Theorem~\ref{thm:decomp} could lead to further developments in this direction.
We likewise anticipate that this decomposition could bear fruit in the study of inhomogeneous equations in the whole space.
\begin{figure}[t]
  \centering
  \includegraphics[width=\linewidth]{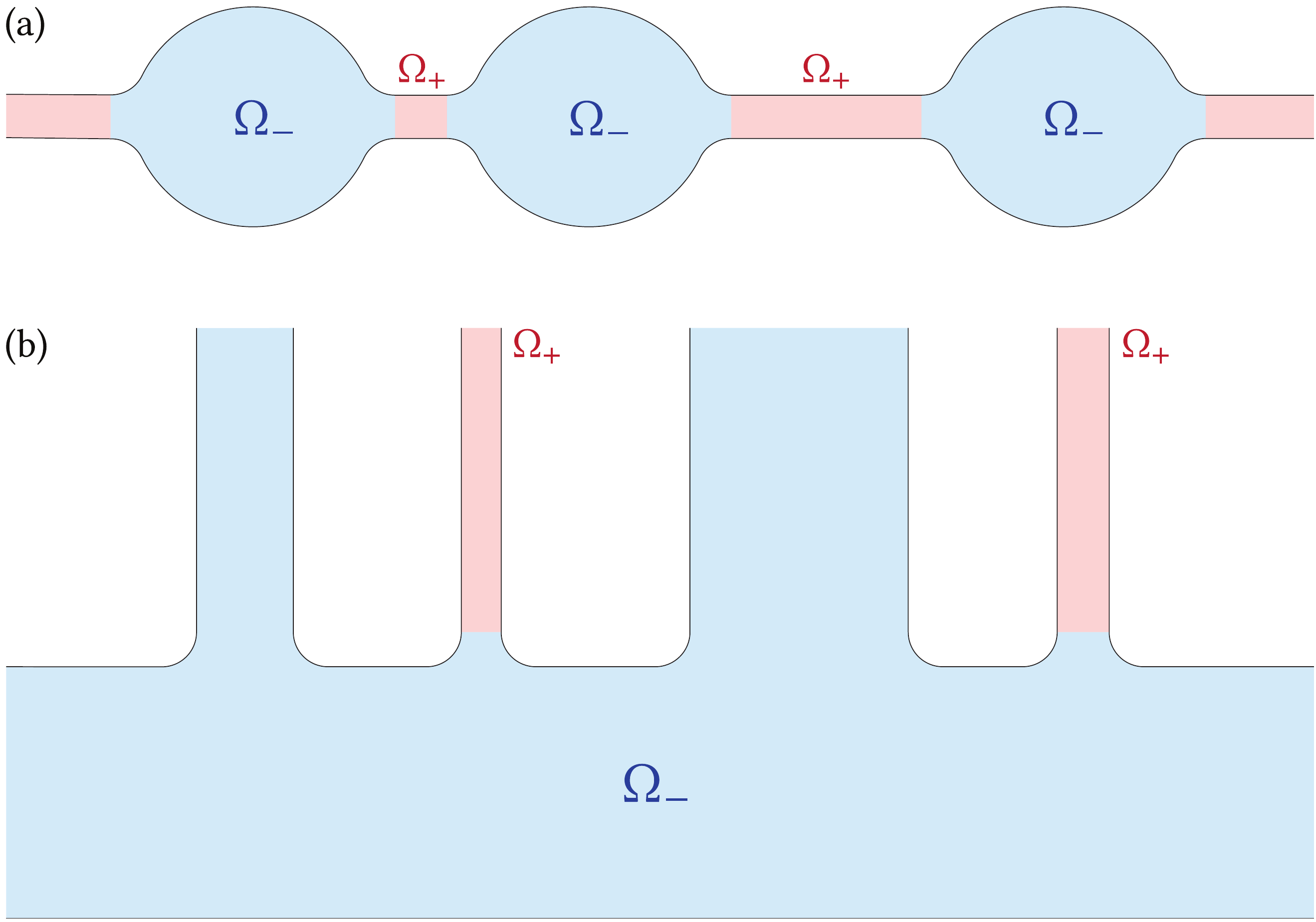}
  \vspace*{2pt}
  \caption[Examples of the ample--narrow decomposition]{
    Illustrations of the decomposition in Theorem~\ref{thm:decomp} with $\Omega_-$ in blue and $\Omega_+$ in red.
    Strictly speaking, the sets are smooth and overlap; we omit these features for clarity.
    (a) An infinite sequence of large balls joined by narrow tubes of varying lengths.
    (b) An infinite ``comb'' with varying tooth width.
    As these examples show, both $\Omega_+$ and $\Omega_-$ may be disconnected with infinitely many components.
  }
  \label{fig:ample-narrow}
\end{figure}

At present, our methods cannot treat critical eigenvalues approaching or equal to $f'(0)$.
Nonetheless, we conjecture that strong-KPP uniqueness holds in general.
\begin{conjecture}
  \label{conj:KPP}
  In a general domain, if $f$ is strong-KPP, then \eqref{eq:main} admits at most one positive bounded solution.
\end{conjecture}
To support this conjecture, we prove it in the special case of a bounded domain augmented by a half-cylinder with a critical limit eigenvalue; see Section~\ref{sec:bulb} for details.

The strong KPP property~\ref{hyp:strong} is essential to the uniqueness in Theorem~\ref{thm:uniqueness}.
In fact, under Dirichlet or Robin conditions, weak-KPP reactions can have multiple positive solutions even on bounded domains.
\begin{proposition}
  \label{prop:nonuniqueness}
  For every bounded domain $\Omega$ and boundary parameter $\const \in [0, 1)$, there exists a reaction $f$ satisfying \ref{hyp:weak} such that \eqref{eq:main} admits multiple positive bounded solutions.
\end{proposition}
On the other hand, under Neumann conditions ($\const = 1$), there is no absorption in the system.
Then we expect the reaction $f$ to push every nontrivial steady state to its unique stable root, namely $1.$
The first author, Hamel, and Nadirashvili confirmed this intuition in domains that satisfy a certain geometric condition~\cite{BHN10}.
Using the same approach, Rossi removed this constraint and thereby showed uniqueness for positive reactions in general domains with Neumann boundary---see Corollary~2.8 in \cite{Rossi}.

In the special case that $f$ is strong-KPP, Neumann uniqueness also follows from our Theorem~\ref{thm:uniqueness}.
Indeed, Proposition~\ref{prop:Neumann-zero} below states that $\lambda(\Omega, 1) = 0$ for any uniformly $\m{C}^{2,\gamma}$ domain $\Omega$.
Hence $(\Omega, 1)$ always satisfies the hypothesis of Theorem~\ref{thm:uniqueness}, and the Neumann solution $1$ of \eqref{eq:main} is always unique.
\medskip

We also study existence in \eqref{eq:main}.
In the following, we use the term \emph{periodic} to indicate that $\Omega$ is invariant under all translations in some full-rank lattice.
\begin{theorem}
  \label{thm:existence}
  If $f$ is positive in the sense of \ref{hyp:positive} and $\lambda(\Omega, \const) < f'(0),$ then \eqref{eq:main} admits a positive bounded solution.
  Conversely, if $f$ satisfies the weak KPP hypothesis \ref{hyp:weak} and $\lambda(\Omega, \const) > f'(0),$ then \eqref{eq:main} has no positive bounded solution.
  If, in addition, $f$ is strong-KPP in the sense of \ref{hyp:strong} and $\Omega$ is bounded or periodic, then \eqref{eq:main} has no positive bounded solution when $\lambda(\Omega, \const) = f'(0).$
\end{theorem}
Together with Theorem~\ref{thm:uniqueness}, this completes the classification of bounded steady states in periodic domains.
\begin{corollary}
  \label{cor:periodic}
  If $f$ is strong-KPP and $\Omega$ is periodic, then \eqref{eq:main} has no positive bounded solution when $\lambda(\Omega, \const) \geq f'(0)$ and exactly one such solution when $\lambda(\Omega, \const) < f'(0)$.
\end{corollary}
Theorem~\ref{thm:existence} is analogous to existing results for variable-coefficient operators in $\R^d$; see Theorem 1.3(1) and Theorem 6.1 in \cite{BHRossi}.
In all these results, the critical case of equality is subtle and largely open.
If $\lambda(\Omega, \const) = f'(0)$ and $f(s) \equiv f'(0)s$ in a nontrivial interval containing $0$, then \eqref{eq:main} may admit positive bounded solutions in the form of principal eigenfunctions.
In this case, any small multiple of the eigenfunction solves \eqref{eq:main}, so uniqueness fails.
However, the strong KPP hypothesis precludes this structure, and in bounded domains with $\lambda(\Omega, \varrho) = f'(0)$, there are no positive bounded solutions.
We expect unbounded critical domains to exhibit the same behavior:
\begin{conjecture}
  \label{conj:critical-nonexistence}
  If $f$ satisfies \ref{hyp:strong} and $\lambda(\Omega, \const) = f'(0),$ then \eqref{eq:main} has no positive bounded solution.
\end{conjecture}
When \eqref{eq:main} admits a unique positive bounded solution $u$, we expect nontrivial solutions of the corresponding parabolic problem to converge to $u$ at large times.
Let $v \in \m{C}^2([0, \infty) \times \Omega; \R)$ solve
\begin{equation}
  \label{eq:parabolic}
  \begin{cases}
    \partial_t v = \Delta v + f(v) &\text{in } \Omega,\\
    \m{N}_\const v = 0 &\text{on } \partial \Omega\\
  \end{cases}
\end{equation}
with strong-KPP $f$.
Then $v$ satisfies the so-called ``hair-trigger effect.''
\begin{proposition}
  \label{prop:hair-trigger}
  Suppose $f$ is positive in the sense of \ref{hyp:positive} and $\lambda(\Omega, \const) < f'(0)$.
  Then \eqref{eq:main} admits a minimal positive bounded solution $u$ such that every positive bounded solution $\ti u$ satisfies $\ti u \geq u$.
  Moreover, if $v$ solves \eqref{eq:parabolic} with nonnegative bounded initial data $v(0, \anon) \not\equiv 0$, then
  \begin{equation*}
    \liminf_{t \to \infty} v(t, \anon) \geq u.
  \end{equation*}
  Suppose, in addition, that $f$ is strong-KPP in the sense of \ref{hyp:strong} and $f'(0) \not \in \bar\Sigma(\Omega, \const)$, so $u$ is the unique positive bounded solution.
  Then
  \begin{equation*}
    v(t, \anon) \to u
  \end{equation*}
  locally uniformly in $\bar\Omega$ as $t \to \infty$.
\end{proposition}
This proposition raises the question of the asymptotic speed of propagation: how rapidly does the disturbance spread through the domain?
This matter was addressed in \cite{BHN10} for KPP reactions with Neumann boundary conditions.
The nature of spreading in general domains with Dirichlet or Robin conditions remains open.

\subsection*{Related works.}
As indicated earlier, our investigation is part of a very broad study of reaction--diffusion dynamics.
Here, we review a selection of works that pertain to our own.
Naturally, our discussion is by no means comprehensive.

\subsubsection*{Homogeneous problems}
In this paper, we consider steady states in subsets of $\R^d$.
In $\R^d$ itself, $1$ is a solution due to the absence of boundary.
For positive reactions, the ``hair-trigger effect'' first observed by Aronson and Weinberger~\cite{AW} ensures that $1$ is the unique positive bounded steady state.
Other reaction classes can exhibit far more elaborate behavior.
For instance, steady states of bistable reactions on the whole space are the subject of the celebrated de Giorgi conjecture~\cite{deGiorgi}.
See~\mbox{\cite{Savin, DPKW}} for landmark results in this direction.

\subsubsection*{Bounded domains}
 Rabinowitz~\cite{Rabinowitz} showed uniqueness for equations of the form $-\m{L}u = f(u)$ for self-adjoint $\m{L}$ and strong-KPP $f$ on bounded domains.
His argument pervades this paper.
The first author later extended uniqueness to non-self-adjoint operators~\cite{Berestycki}.
We emphasize that in such generality, the result is false in unbounded domains.
To see this, recall that KPP equations on the line admit traveling wave solutions that move with positive speed.
If we shift into a co-moving frame, such a wave becomes a positive bounded steady-state in addition to $1$.
Crucially, the shifted operator includes a first-order term and is not self-adjoint.
It seems likely that some form of Theorem~\ref{thm:uniqueness} remains true for more general self-adjoint operators.
For further discussion along this line, see \cite{BHRossi}.

\subsubsection*{Unbounded domains}
The study of steady states on unbounded domains has largely focused on domains with particular structure.
In extensive collaboration with Caffarelli and Nirenberg, the first author has examined qualitative properties of elliptic solutions on half-spaces~\cite{BCN93}, cylinders~\cite{BCN96}, and epigraphs~\cite{BCN97b}.
Although uniqueness was not the principal focus of these works, \cite{BCN96} establishes uniqueness for strong-KPP reactions on cylinders and \cite{BCN97b} shows uniqueness for positive reactions on uniformly Lipschitz epigraphs.
We also highlight a collaboration~\cite{BR} of Rossi and the first author, which analyzes the linear problem in general unbounded domains.
We draw heavily on characterizations of the generalized principal eigenvalue from~\cite{BR}.

Under Neumann boundary conditions, the classification of steady states is somewhat simpler, as $1$ is a bounded solution of \eqref{eq:main}.
Several prior works have touched on Neumann uniqueness with weak-KPP~\cite{BHN10} or positive reactions~\cite{Rossi}.
As we discussed earlier, the former treats domains satisfying a certain geometric constraint and the latter removes the constraint, thereby proving uniqueness on general (possibly unbounded) domains.

\subsubsection*{Inhomogeneous equations}
Boundary conditions are one way to model the effects of inhomogeneous environments.
Of course, one can also study inhomogeneous \emph{equations} on the whole space.
In collaboration with Hamel and Rossi, the first author has studied the uniqueness of steady states in such systems~\cite{BHRossi}.
This work is a major inspiration for our own; it establishes uniqueness under certain conditions on the generalized principal eigenvalues of limiting problems.
Their conditions are more restrictive than ours, however.
In our language, the main Theorem~1.3(2) of \cite{BHRossi} is analogous to uniqueness when $\Sigma(\Omega, \const) \subset [0, f'(0)).$
Thus \cite{BHRossi} does not discuss systems with limit eigenvalues both above and below $f'(0)$.
It would be interesting to see whether a more general condition like $f'(0) \not \in \bar\Sigma(\Omega, \const)$ could ensure uniqueness in the variable-coefficient setting.

\subsubsection*{Homogeneous evolution}
As noted earlier, one can view solutions of \eqref{eq:main} as steady states of time-dependent reaction--diffusion equations.
Transitions from one steady state to another have featured prominently in the theory of such parabolic equations.
In the homogeneous KPP case, $0$ is unstable and $1$ is stable.
Any initial disturbance will trigger a transition from $0$ to $1$ that spreads in space at the asymptotic speed of propagation introduced by Aronson and Weinberger~\cite{AW}.

\subsubsection*{Evolution in periodic media}
The parabolic theory of inhomogeneous systems is dominated by the study of periodic equations in periodic domains.
In such systems, the asymptotic speed of propagation may depend on the direction.
These speeds were characterized in increasing generality by Freidlin and G\"artner~\cite{FG} and Weinberger~\cite{Weinberger}.
Hamel, Nadirashvili, and the first author have examined detailed properties of the speeds for weak-KPP reactions~\cite{BHN05}.
In collaboration with Hamel and Roques~\cite{BHRoques-1}, the first author has studied steady states and spreading properties in periodic media. 
In this case, the strong-KPP condition implies the uniqueness of steady states.

\subsubsection*{General inhomogeneous evolution}
The literature on aperiodic systems is comparatively sparse.
In~\cite{BHN10}, Hamel, Nadirashvili, and the first author consider the asymptotic speed of propagation of KPP reactions in general domains with Neumann boundary data.
Under geometric conditions on the domain, they show that $1$ is the unique steady state and that solutions propagate at the usual linearly-determined speed in ``free'' directions that avoid the boundary.
Because it treats Neumann data, this work does not incorporate the boundary absorption that we emphasize in the present paper.

A handful of results are known for specific domains with Dirichlet data.
Luo and Lu~\cite{LL} have established the uniqueness of steady states and the asymptotic speed of propagation for KPP reactions on cones.
For analogous results on half-spaces with a variety of boundary conditions and reactions, see our prior work~\cite{BG}.

Hamel, Nadin, and the first author have studied general time-dependent inhomogeneous reactions in the whole space~\cite{BHNadin}.
When the reaction is strong-KPP, they prove the uniqueness of a uniformly-positive entire solution that is analogous to our steady states.
In a further collaboration, Nadin and the first author have extensively studied weak-KPP spreading in quite general inhomogeneous media~\cite{BNadin}.

There is also a growing body of work on spreading in random media, which falls within the broader theory of stochastic homogenization.
For results in this direction, we direct the reader to work of Lin and Zlato\v{s}~\cite{LZ} and the references therein.

\subsection*{Organization.}
The present paper is organized as follows.
In Section~\ref{sec:lambda}, we establish fundamental properties of the generalized principal eigenvalue $\lambda(\Omega, \const)$.
Drawing on these preliminaries, we prove a maximum principle and Theorem~\ref{thm:existence} in Section~\ref{sec:existence}.
We devote Section~\ref{sec:uniqueness} to the decomposition in Theorem~\ref{thm:decomp} and our uniqueness results Theorem~\ref{thm:uniqueness} and Proposition~\ref{prop:nonuniqueness}.
Finally, we treat a special example with critical limit eigenvalue in Section~\ref{sec:bulb}.
An appendix puts our notions of smoothness and convergence for domains on rigorous footing.

\section{Properties of the principal eigenvalue}
\label{sec:lambda}

In this section, we discuss the essential properties of the generalized principal eigenvalue.
We consider elliptic operators of the form
\begin{equation}
  \label{eq:elliptic}
  \m{L} \coloneqq \Delta + q(x)
\end{equation}
with a uniformly $\m{C}^\gamma$ potential $q$.
We let $\m{C}_{\mathrm{b}}^\gamma(\Omega)$ denotes the space of uniformly $\m{C}^\gamma$ functions on $\Omega$.
This naturally generalizes to $\m{C}_{\mathrm{b}}^{k,\al}(\Omega)$.

In the introduction, we confined our attention to constant boundary conditions.
However, as we cut our domains into pieces, we may employ Robin conditions on some parts of the boundary and Dirichlet conditions on others.
We encompass such mixed boundary conditions through variable boundary parameters ${\var \colon \partial \Omega \to [0, 1]}$.
Unless otherwise stated, we will assume that $\var$ is uniformly $\m{C}^{2,\gamma}$.
We often abbreviate and refer to the pair $(\Omega, \var)$ as a uniformly $\m{C}^{2,\gamma}$ domain.
We let $\m{N}_\var$ denote the corresponding mixed boundary operator as in \eqref{eq:boundary-op}.
Then the definition \eqref{eq:lambda} of the generalized principal eigenvalue extends to mixed boundary conditions.
Throughout the paper, we use $\const$ and $\var$ to denote constant and mixed parameters, respectively.

We often wish to compare $\lambda(\Omega, \var)$ with the principal eigenvalue on a subset ${\Omega' \subset \Omega}$.
Under Dirichlet conditions, this is immediate: $\lambda(\Omega', 0) \geq \lambda(\Omega, 0)$.
Unfortunately, Robin and Neumann conditions do not yield such monotonicity.
Nonetheless, domain monotonicity does hold if we employ suitable boundary conditions on $\Omega'$.
To capture this phenomenon, we define a partial order on pairs $(\Omega, \var)$ of domains and boundary parameters.
\begin{definition}
  \label{def:inclusion}
  We write $(\Omega', \var') \preceq (\Omega, \var)$ if $\Omega' \subset \Omega$, $\var' \leq \var$ on the common boundary $\partial \Omega' \cap \partial \Omega$, and $\var' = 0$ on $\partial \Omega' \cap \Omega$.
\end{definition}
\noindent
We shall see below that $\lambda(\Omega', \var') \geq \lambda(\Omega, \var)$ when $(\Omega', \var') \preceq \lambda(\Omega, \var)$.
Note the reversed order.

As a special case, we frequently consider the intersection between $\Omega$ and a ball $B$.
Motivated by Definition~\ref{def:inclusion}, we impose $\var$-conditions on the ``original'' boundary $\partial \Omega \cap B$ and Dirichlet conditions on the ``artificial'' boundary $\Omega \cap \partial B$.
We use the notation $\lambda(-\m{L}, \Omega \mid B, \var)$ to denote the corresponding generalized principal eigenvalue.
Thus
\begin{equation}
  \label{eq:cut}
  \lambda(-\m{L}, \Omega \mid B, \var) \coloneqq \lambda(-\m{L}, \Omega \cap B, \var \tbf{1}_{\partial \Omega \cap B}).
\end{equation}
Clearly, $(\Omega \cap B, \var \tbf{1}_{\partial \Omega \cap B}) \preceq (\Omega, \var)$ under Definition~\ref{def:inclusion}.

We observe that the domain $\Omega \cap B$ typically has corners and the mixed boundary condition $\var \tbf{1}_{\partial \Omega \cap B}$ is typically discontinuous.
This is our one routine exception to the assumption that $\Omega$ and $\var$ are uniformly $\m{C}^{2,\gamma}$.
Because $\Omega \cap B$ is bounded, the corresponding eigenvalue problem largely falls within established work on low-regularity domains.
We also note that $\Omega \cap B$ may be disconnected.
This is not a major hindrance---our definition of $\lambda$ extends to disconnected domains.
It is straightforward to check that the generalized principal eigenvalue of a disconnected domain is the infimum of the eigenvalues of the connected components.
Also, we adopt the convention that $\lambda = +\infty$ if the domain is empty.

We now turn to characterizations of the generalized principal eigenvalue.
Given $\var \colon \partial \Omega \to [0, 1]$, let $\m{H}_\var$ denote the closure of $\m{C}_0^\infty\big(\bar\Omega \setminus \{\var = 0\}\big)$ in $H^1(\Omega)$.
Informally, $\m{H}_\var$ is the set of $H^1$ functions that vanish on the ``Dirichlet part'' of $\partial \Omega$ where $\var = 0.$
This is the natural space of test functions for a variational formulation of $\lambda$.
In the following proposition, we collect various fundamental properties of $\lambda$.
\begin{proposition}
  \label{prop:lambda}
  Let $(\Omega, \var)$ be a uniformly $\m{C}^{2,\gamma}$ domain.
  Let $\m{L}$ be an elliptic operator of the form \eqref{eq:elliptic} with potential $q \in \m{C}_{\mathrm{b}}^\gamma(\Omega)$.
  Then the following hold.
  \vspace{1ex}
  \begin{enumerate}[label = \textup{(\roman*)}, itemsep = 2ex]
  \item
    \label{item:monotone}
    If $(\Omega', \var') \preceq (\Omega, \var)$, then $\lambda(\Omega', \var') \geq \lambda(\Omega, \var)$.

  \item
    \label{item:trivial-lambda-bd}
    Let $R_{\textnormal{in}}$ denote the (possibly infinite) inradius of $\Omega$.
    Then
    \begin{equation*}
      -\sup q \leq \lambda(-\m{L},\Omega,\var) \leq -\inf q + \lambda(B_1,0) R_{\textnormal{in}}^{-2}.
    \end{equation*}

  \item
    \label{item:exhaustion}
    For every $x \in \R^d$,
    \begin{equation}
      \label{eq:exhaustion}
      \lambda(-\m{L}, \Omega, \var) = \lim_{R \to \infty} \lambda(-\m{L}, \Omega \mid B_R(x), \var).
    \end{equation}

  \item
    \label{item:eigenfunction}
    There exists $\varphi \in \m{C}_{\mathrm{loc}}^{2,\gamma}\big(\bar\Omega\big)$ such that $\varphi > 0$ and
    \begin{equation}
      \label{eq:eigenfunction}
      -\m{L} \varphi = \lambda(-\m{L}, \Omega, \var) \varphi \And \m{N}_\var \varphi = 0.
    \end{equation}

  \item
    \label{item:variation}
    Recalling the space $\m{H}_\sigma$ from above, we have
    \begin{equation}
      \label{eq:variation}
      \lambda(-\m{L}, \Omega, \var) = \inf_{\phi \in \m{H}_\var \setminus \{0\}} \frac{\int_{\Omega} \left[\abs{\nab \phi}^2 - q \phi^2\right] + \int_{\partial \Omega} \var^{-1}(1 - \var) \phi^2}{\int_{\Omega} \phi^2}.
    \end{equation}

  \item
    \label{item:bounded}
    If $\Omega$ is bounded, then $\lambda(-\m{L}, \Omega, \var)$ is the classical principal eigenvalue of $-\m{L}$ on $\Omega$ with boundary parameter $\var$.
  \end{enumerate}
\end{proposition}
\noindent
These properties are already known in the Dirichlet setting; see Proposition~2.3 of \cite{BR} and the references therein.
\begin{proof}
  \noindent
  \textbf{\ref{item:monotone}.}
  By Definition~\ref{def:inclusion}, a supersolution for $\lambda(\Omega,\const)$ in the sense of \eqref{eq:lambda} is also a supersolution for $\lambda(\Omega', \const')$.
  \smallskip

  \noindent
  \textbf{\ref{item:trivial-lambda-bd}.}
  The constant supersolution $\phi = 1$ in \eqref{eq:lambda} yields the lower bound.
  For the upper bound, let $B \subset \Omega$ be a ball of radius $R > 0$.
  Using \eqref{eq:lambda}, \ref{item:monotone}, and scaling, we find
  \begin{equation*}
    \lambda(-\m{L},\Omega,\var) \leq \lambda(-\m{L},B,0) \leq -\inf q + \lambda(B,0) = -\inf q + \lambda(B_1,0) R^{-2}.
  \end{equation*}
  Varying over all inscribed balls, we obtain \ref{item:trivial-lambda-bd}.
  \smallskip

  \noindent
  \textbf{\ref{item:exhaustion} and \ref{item:eigenfunction}.}
  By \ref{item:monotone}, the sequence $\lambda(-\m{L}, \Omega \mid B_R(x), \var)$ is decreasing in $R$ and
  \begin{equation}
    \label{eq:liminf-lambda}
    \lambda(-\m{L}, \Omega, \var) \leq \lim_{R \to \infty} \lambda(-\m{L}, \Omega \mid B_R(x), \var).
  \end{equation}
  In the Dirichlet case $\var \equiv 0$, \cite{BNV} shows the existence of a principal eigenfunction $\varphi_R > 0$ corresponding to $\lambda(-\m{L}, \Omega \mid B_R(x), \var)$, despite the possible corners in $\Omega \cap B_R$.
  Similar arguments apply in the Robin case; we omit the proof.
  Precisely, there exists $\varphi_R$ on $\Omega \cap B_R(x)$ such that
  \begin{equation*}
    -\m{L} \varphi_R = \lambda(-\m{L}, \Omega \mid B_R(x), \var) \varphi_R \quad \text{in } \Omega \cap B
  \end{equation*}
  and
  \begin{equation*}
    \m{N}_{\var \tbf{1}_{\partial \Omega \cap B_R(x)}} \varphi_R = 0 \quad \text{ almost everywhere on } \partial(\Omega \cap B_R(x)).
  \end{equation*}
  Assume without loss of generality that $0 \in \Omega$ and $R \geq \abs{x} + 1$ so that $0 \in \Omega \cap B_R(x)$.
  We normalize $\varphi_R$ by $\varphi_R(0) = 1$.
  Fix $K \Subset \bar{\Omega}$ and $R_K > 0$ such that $K \subset B_{R_K - 1}(x)$.
  Then Harnack and Schauder estimates up to the boundary imply that the family $(\varphi_R)_{R \geq R_K}$ is uniformly $\m{C}^{2,\gamma}$ smooth on $K$.
  Fix $\al \in (0, \gamma)$.
  By Arzel\`a--Ascoli and diagonalization, we can extract a subsequence of radii tending to infinity along which $\varphi_R$ converges locally uniformly in $\m{C}^{2,\al}$ to a limit $\varphi \in \m{C}^{2,\gamma}(\Omega)$.
  We can easily check that $\varphi$ is an eigenfunction in the sense of \eqref{eq:eigenfunction} with eigenvalue
  \begin{equation*}
    \lim_{R \to \infty} \lambda(-\m{L}, \Omega \mid B_R(x), \var).
  \end{equation*}
  Since $\varphi(0) = 1$ and $\varphi \geq 0$, the strong maximum principle implies that $\varphi > 0$ in $\Omega$.
  Using $\varphi$ as a supersolution in \eqref{eq:lambda}, we see that
  \begin{equation*}
    \lambda(-\m{L}, \Omega, \var) \geq \lim_{R \to \infty} \lambda(-\m{L}, \Omega \mid B_R(x), \var).
  \end{equation*}
  In combination with \eqref{eq:liminf-lambda}, this proves \eqref{eq:exhaustion} and \ref{item:eigenfunction}.
  \smallskip

  \noindent
  \textbf{\ref{item:variation}.}
  The variational formulation \eqref{eq:variation} is classical when $\Omega$ is bounded.
  The general form follows from \ref{item:exhaustion} and a standard approximation argument.
  \smallskip

  \noindent
  \textbf{\ref{item:bounded}.}
  It is well known that the classical principal eigenvalue is given by the Rayleigh quotient \eqref{eq:variation}, so \ref{item:bounded} follows from \ref{item:variation}.
\end{proof}
Many of our key arguments rely on various manipulations of the domain $\Omega$.
We often take limits or intersect domains with balls as in \eqref{eq:cut}.
The next two results control $\lambda(\Omega, \var) = \lambda(-\Delta, \Omega, \var)$ under these operations.

The precise sense in which a sequence of domains converges to another is technical and has a distinct flavor from the rest of the paper.
We therefore collect such details in Appendix~\ref{sec:smooth}.
Definition~\ref{def:limit} establishes our notion of locally uniform $\m{C}^{2,\al}$ convergence $(\Omega_n, \var_n) \to (\Omega_\infty, \var_\infty)$.
We supplement this definition with a corresponding compactness result.
Corollary~\ref{cor:compactness} ensures that any uniformly $\m{C}^{2,\gamma}$ sequence $(\Omega_n, \var_n)$ has a subsequence that converges in this sense (provided $\al < \gamma$).

We now turn to the notion of \emph{connected} limits.
Given $x \in \R^d$, define the translation operator
\begin{equation*}
  \tau_x \Omega \coloneqq \Omega - x \And \tau_x \var \coloneqq \var(\anon + x).
\end{equation*}
Then we define connected limits as follows.
\begin{definition}
  \label{def:connected-limit}
  Let $(\Omega_n, \var_n)_{n \in \N}$ be a uniformly $\m{C}^{2, \gamma}$ sequence of domains such that $0 \in \bar{\Omega}_n$.
  We say that $(\Omega^*, \var^*)$ is the \emph{connected limit} of $(\Omega_n, \var_n)_{n \in \N}$ if the sequence converges locally uniformly in $\m{C}^{2, \al}$ for some $\al \in (0, \gamma)$ to a domain $(\Omega_\infty, \var_\infty)$, $\Omega^*$ is the connected component of $\Omega_\infty$ such that $0 \in \bar \Omega^*$, and $\var^* = \var_\infty|_{\partial \Omega^*}$.
  We also say that $(\Omega^*, \var^*)$ is the connected limit of a single domain $(\Omega, \var)$ \emph{along } $(x_n)_{n \in \N} \subset \Omega$ if $(\Omega^*, \var^*)$ is the connected limit of $(\tau_{x_n}\Omega, \tau_{x_n}\var)_{n \in \N}$.
\end{definition}
If $(x_n)_{n \in \N}$ is \emph{any} sequence in $\Omega$, then the uniform $\m{C}^{2,\gamma}$-smoothness of $\Omega$ and Corollary~\ref{cor:compactness} ensure that the domains $(\tau_{x_n}\Omega, \tau_{x_n}\var)_{n \in \N}$ have a nonempty uniformly $\m{C}^{2,\gamma}$ subsequential limit $(\Omega^*, \var^*)$.
We note that the uniform interior ball condition on $\Omega$ prevents the limit from degenerating to $\emptyset$.

We extend the definition of the principal limit spectrum to $\Sigma(\Omega, \var)$ in the natural manner.
Note that every translate of $(\Omega, \var)$ to the origin is a connected limit along a constant sequence.
Therefore
\begin{equation}
  \label{eq:identity}
  \lambda(\Omega, \var) \in \Sigma(\Omega, \var).
\end{equation}
We now constrain the behavior of $\lambda$ under connected limits.
\begin{lemma}
  \label{lem:semicontinuous}
  Suppose $(\Omega^*, \var^*)$ is the connected limit of a uniformly $\m{C}^{2,\gamma}$ sequence $(\Omega_n, \var_n)_{n \in \N}$.
  Then
  \begin{equation}
    \label{eq:semicontinuous}
    \lambda(\Omega^*, \var^*) \geq \limsup_{n \to \infty} \lambda(\Omega_n, \var_n).
  \end{equation}
\end{lemma}
\noindent
That is, the generalized principal eigenvalue is upper semicontinuous in $(\Omega, \var)$.
Lemma~\ref{lem:semicontinuous} is analogous to Proposition~5.3 in~\cite{BHRossi}, and the proof is similar.
The proof also recalls that of Proposition~\ref{prop:lambda}\ref{item:exhaustion} above.
\begin{figure}[t]
  \centering
  \includegraphics[width=0.8\linewidth]{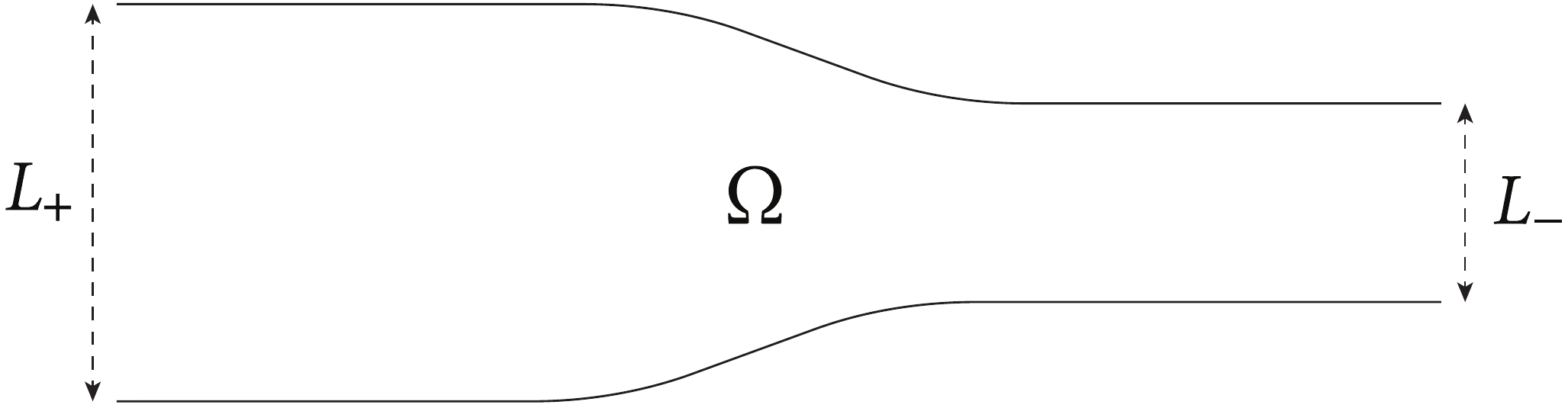}
  \caption{A narrowing strip.}
  \label{fig:narrowing-tube}
\end{figure}
\begin{remark}
  \label{rem:continuous}
  In general, this result cannot be improved---volume can ``escape to infinity'' and vanish in the limit, leading to strict inequality in \eqref{eq:semicontinuous}.
  For example, consider the narrowing strip $\Omega$ in Figure~\ref{fig:narrowing-tube}.
  Using Proposition~\ref{prop:lambda}, one can check that $\lambda(\Omega, 0) = (\pi/L_+)^2$.
  This principal eigenvalue is determined by the half-cylinder on the left of width $L_+$.

  Let us translate the domain steadily to the left via ${\Omega_n \coloneqq \Omega - (n, 0)}$.
  The principal eigenvalue is translation-invariant, so $\lambda(\Omega_n, 0) = (\pi/L_+)^2$ for all $n \in \N$.
  However, in the $n \to \infty$ limit we are left with a narrower cylinder $\Omega^*$ of width $L_- < L_+$.
  This cylinder satisfies $\lambda(\Omega_*, 0) = (\pi/L_-)^2$, so
  \begin{equation*}
    \lambda(\Omega^*, 0) > \limsup_{n \to \infty} \lambda(\Omega_n, 0).
  \end{equation*}

  Conversely, if the sequence $(\Omega_n)_n$ of domains is uniformly bounded, then volume cannot escape to infinity.
  In this case, equality holds in \eqref{eq:semicontinuous}.
  That is, $\lambda_1$ is continuous in $(\Omega, \sigma)$ when $\Omega$ is uniformly bounded.
  This continuity holds in significantly weaker topologies; see, for instance, \cite{Butazzo}.
\end{remark}
\begin{proof}[Proof of Lemma~\textup{\ref{lem:semicontinuous}}]
  By Definition~\ref{def:limit}, the connected limit $\Omega^*$ must share a point in common with all $\Omega_n$ for sufficiently large $n$.
  Dropping smaller $n \in \N$ and shifting, we are free to assume that $0 \in \bigcap_{n} \Omega_n \cap \Omega^*$.
  By Proposition~\ref{prop:lambda}\ref{item:eigenfunction}, each $\Omega_n$ admits a principal eigenfunction $\varphi_n > 0$ such that $\varphi_n(0) = 1$.
  By Corollary~\ref{cor:compactness}, we can extract a subsequence converging to an eigenfunction $\varphi^*$ on $\Omega^*$ with eigenvalue
  \begin{equation*}
    \limsup_{n \to \infty} \lambda(\Omega_n, \var_n).
  \end{equation*}
  Since $\varphi^*(0) = 1$ and $\varphi^* \geq 0$, the strong maximum principle implies that $\varphi^* > 0$.
  Using $\varphi^*$ as a supersolution in \eqref{eq:lambda}, we see that
  \begin{equation*}
    \lambda(\Omega^*, \var^*) \geq \limsup_{n \to \infty} \lambda(\Omega_n, \var_n).\qedhere
  \end{equation*}
\end{proof}
Next, we establish a Robin version of a beautiful result of Lieb.
The following is an elementary extension of Theorem~1.1 in \cite{Lieb} to non-Dirichlet boundary conditions.
\begin{theorem}
  \label{thm:Lieb}
  Let $(\Omega, \var)$ be uniformly $\m{C}^{2,\gamma}$.
  Then for all $R > 0$,
  \begin{equation*}
    \inf_{x \in \R^d} \lambda(\Omega \mid B_R(x), \var) \leq \lambda(\Omega, \var) + \lambda(B_1, 0) R^{-2}.
  \end{equation*}
\end{theorem}
This is an essential tool in our analysis of the principal eigenvalue.
Informally, it states that $\lambda$ is local.
To determine $\lambda$ to accuracy $\eps$, it suffices to examine $\Omega$ at scale $\propto \eps^{-2}$.
\begin{proof}
  Fix $\eps > 0$.
  By Proposition~\ref{prop:lambda}\ref{item:variation} and the definition of $\m{H}_\var$, there exist $f \in \m{C}_0^\infty\big(\bar\Omega \setminus \{\var = 0\}\big)$ and $g \in \m{C}_0^\infty(B_R)$  with unit $L^2$ norms such that
  \begin{equation}
    \label{eq:almost-eigen-f}
    \int_{\Omega} \abs{\nab f}^2 + \int_{\partial \Omega} \var^{-1}(1 - \var) f^2 < \lambda(\Omega, \var) + \eps/2
  \end{equation}
  and
  \begin{equation}
    \label{eq:almost-eigen-g}
    \int_{B_R} \abs{\nab g}^2 < \lambda(B_R, 0) + \eps/2.
  \end{equation}
  We extend $g$ by $0$ to $\R^d$ and define
  \begin{equation*}
    h_x(y) \coloneqq f(y) g(y - x)
  \end{equation*}
  for $x \in \R^d$ and $y \in \Omega$.
  Then $h_x \in \m{H}_\var$.
  Next, we define
  \begin{equation*}
    T(x) \coloneqq \int_{\Omega} \abs{\nab h_x}^2 + \int_{\partial \Omega} \var^{-1}(1 - \var) h_x^2 \And D(x) \coloneqq \int_\Omega h_x^2.
  \end{equation*}
  Fubini and our $L^2$-normalization of $f$ and $g$ imply that $\int_{\R^d} D(x) = 1$.
  For $T$, we can compute
  \begin{equation*}
    \abs{\nab h_x(y)}^2 = \abs{\nab f(y)}^2 g(y - x)^2 + f(y)^2 \abs{\nab g(y - x)}^2 + \frac{1}{2} \nab (f^2)(y) \cdot \nab(g^2)(y - x).
  \end{equation*}
  Writing $ \nab(g^2)(y - x)$ as $-\nab_x[g(y - x)^2]$, we see that the final term vanishes when we integrate in $x$.
  Thus by Fubini,
  \begin{equation}
    \label{eq:deriv-int}
    \begin{aligned}
      \int_{\R^d \times \Omega} \abs{\nab h_x}^2 &= \int_{\R^d \times \Omega} \left[\abs{\nab f(y)}^2 g(y - x)^2 + f(y)^2 \abs{\nab g(y - x)}^2\right]\\
      &= \int_\Omega \abs{\nab f}^2 + \int_{B_R} \abs{\nab g}^2.
    \end{aligned}
  \end{equation}
  Also,
  \begin{equation}
    \label{eq:boundary-int}
    \int_{\R^d \times \partial\Omega} \var^{-1}(1 - \var) h_x^2 = \int_{\partial\Omega} \var^{-1}(1 - \var) f(y)^2.
  \end{equation}
  Combining \eqref{eq:deriv-int} and \eqref{eq:boundary-int} with \eqref{eq:almost-eigen-f} and \eqref{eq:almost-eigen-g}, we see that
  \begin{equation*}
    \int_{\R^d} T(x) < \lambda(\Omega, \var) + \lambda(B_R, 0) + \eps.
  \end{equation*}
  Since $\int D = 1$, we have
  \begin{equation*}
    \int_{\R^d} \left[T(x) - (\lambda(\Omega, \var) + \lambda(B_R, 0) + \eps) D(x)\right] < 0.
  \end{equation*}
  Therefore
  \begin{equation}
    \label{eq:positive-set}
    (\lambda(\Omega, \var) + \lambda(B_R, 0) + \eps) D(x) > T(x)
  \end{equation}
  on a set of positive measure.
  In particular, there exists $x \in \R^d$ satisfying \eqref{eq:positive-set}.
  If we substitute $h_x$ into the Rayleigh quotient in Proposition~\ref{prop:lambda}\ref{item:variation}, we obtain $T(x)/D(x)$.
  Thus by our choice of $x$, we see that
  \begin{equation*}
    \lambda(\Omega \mid B_R(x), \var) \leq \lambda(\Omega, \var) + \lambda(B_R, 0) + \eps = \lambda(\Omega, \var) + \lambda(B_1, 0) R^{-2} + \eps.
  \end{equation*}
  Since $\eps > 0$ was arbitrary, the theorem follows.
\end{proof}
Next, we establish the continuity of $\lambda$ in the boundary parameter.
We do not make use of this result, but it may be of independent interest.
The proof is a typical application of Lieb's theorem.
\begin{proposition}
  Let $\Omega$ be uniformly $\m{C}^{2,\gamma}$.
  Suppose $(\var_n)_{n \in \N}$ is a uniformly $\m{C}^{2,\gamma}$ sequence of boundary parameters converging in $L^\infty(\partial \Omega)$ to a $\m{C}_{\mathrm{b}}^{2,\gamma}$ function $\var$.
  Then
  \begin{equation*}
    \lim_{n \to \infty} \lambda(\Omega, \var_n) \to \lambda(\Omega, \var).
  \end{equation*}
\end{proposition}
\begin{proof}
  By Proposition~\ref{prop:lambda}\ref{item:monotone}, $\lambda(\Omega, \var)$ is a decreasing function of $\var$.
  We are therefore free to assume that the sequence $(\var_n)_{n \in \N}$ is monotone.
  First suppose $\var_n \nearrow \var$.
  Without loss of generality, we may assume that $0 \in \Omega$.
  By Proposition~\ref{prop:lambda}\ref{item:eigenfunction}, for each $n$, there exists a $\var_n$-principal eigenfunction $\varphi_n$ with $\varphi_n(0) = 1$.
  Elliptic estimates allow us to extract a locally convergent subsequence with limit $\varphi$, which is a positive eigenfunction with boundary parameter $\var$ and eigenvalue
  \begin{equation*}
    \lim_{n \to \infty} \lambda(\Omega, \var_n).
  \end{equation*}
  Using $\varphi$ as a supersolution in \eqref{eq:lambda}, we see that
  \begin{equation*}
    \lambda(\Omega, \var) \geq \lim_{n \to \infty} \lambda(\Omega, \var_n).
  \end{equation*}
  The reverse inequality follows from $\var_n \nearrow \var$ and Proposition~\ref{prop:lambda}\ref{item:monotone}, so in fact
  \begin{equation*}
    \lambda(\Omega, \var) = \lim_{n \to \infty} \lambda(\Omega, \var_n).
  \end{equation*}

  Now suppose $\var_n \searrow \var$.
  This case is more subtle.
  Fix $\eps > 0$ and choose $R > 0$ such that $\lambda(B_1, 0) R^{-2} < \eps$.
  For each $n \in \N$, Theorem~\ref{thm:Lieb} provides $x_n \in \R^d$ such that
  \begin{equation}
    \label{eq:good-ball}
    \lambda(\Omega \mid B_R(x_n), \var_n) \leq \lambda(\Omega, \var_n) + \eps.
  \end{equation}
  The point $x_n$ need not lie in $\Omega$, but it is no more than distance $R$ away from $\Omega$, since $B_R(x_n) \cap \Omega$ must be nonempty.
  Let $y_n$ be a point in the component of $\Omega \cap B_R(x_n)$ with minimal eigenvalue.
  Taking $n \to \infty$, we can extract a nonempty connected limit $(\Omega^*, \var^*)$ of $(\Omega, \var_n)$ along $(y_n)_{n \in \N}$.
  The sequence $(x_n - y_n)_n$ lies in $B_R$, so we can extract a further limit point $z$ of $(x_n - y_n)_n$.
  We note that $B_R(z) \cap \Omega^*$ is nonempty, for otherwise $\lambda(\Omega \mid B_R(x_n), \var_n)$ would tend to infinity, contradicting \eqref{eq:good-ball}.
  Restricting to a subsequence, there exist positive eigenfunctions $\psi_n$ corresponding to $\lambda(\Omega \mid B_R(x_n), \var_n)$ such that $\tau_{y_n}\psi_n$ converges as $n \to \infty$ to a positive eigenfunction $\psi^*$ for $(\Omega^* \mid B_R(z), \var^*)$ with eigenvalue
  \begin{equation*}
    \lim_{n \to \infty} \lambda(\Omega \mid B_R(x_n), \var_n).
  \end{equation*}
  Because $\Omega^* \cap B_R$ is bounded, every positive eigenfunction is principal, so in fact
  \begin{equation}
    \label{eq:principal-limit}
    \lambda(\Omega^* \mid B_R(z), \var) = \lim_{n\to\infty}\lambda(\Omega \mid B_R(x_n), \var_n).
  \end{equation}
  Therefore Lemma~\ref{lem:semicontinuous}, \eqref{eq:good-ball}, and \eqref{eq:principal-limit} yield
  \begin{align*}
    \lambda(\Omega, \var) \leq \lambda(\Omega^*, \var^*) \leq \lambda(\Omega^* \mid B_R(z), \var^*) &= \lim_{n\to\infty} \lambda(\Omega \mid B_R(x_n), \var_n)\\
                                                                                                 &\leq \lim_{n\to\infty} \lambda(\Omega, \var_n) + \eps.
  \end{align*}
  Since $\eps > 0$ was arbitrary,
  \begin{equation*}
    \lambda(\Omega, \var) \leq \lim_{n\to\infty} \lambda(\Omega, \var_n).
  \end{equation*}
  The reverse inequality follows from $\var_n \searrow \var$ and Proposition~\ref{prop:lambda}\ref{item:monotone}.
\end{proof}
To close the section, we show that the Neumann generalized principal eigenvalue is zero.
Our argument is similar to the proof of Theorem~2.3 in \cite{BHN10}.
\begin{proposition}
  \label{prop:Neumann-zero}
  If $\Omega$ is uniformly $\m{C}^{2,\gamma}$, then $\lambda(\Omega, 1) = 0$.
\end{proposition}
\begin{proof}
  Without loss of generality, assume $0 \in \Omega$.
  Given $R > 1$, we abbreviate $\Omega \cap B_R$ by $\Omega_R$.
  We control $\lambda(\Omega, 1)$ via the variational formula \eqref{eq:variation}.
  Let $\op{dist}_\Omega$ denote the distance induced on the metric space $\Omega$ through its inclusion in $\R^d$, and define the test function
  \begin{equation*}
    \psi_R(y) \coloneqq \min\big\{\op{dist}_\Omega(y, \partial \Omega_R \cap \Omega), 1\big\}.
  \end{equation*}
  Then
  \begin{alignat}{3}
    \abs{\nab \psi} &\leq \tbf{1}_{\Omega_R \setminus \Omega_{R - 1}} \quad &&\text{in } \Omega_R,\label{eq:psi-grad-bd}\\
    \psi &= 1 &&\text{in } \Omega_{R - 1},\label{eq:psi-const}\\
    \psi &= 0 &&\text{on } \partial \Omega_R \cap \Omega.\nonumber
  \end{alignat}
  Recall the space of test functions $\m{H}_1 = H^1(\Omega)$ defined before Proposition~\ref{prop:lambda}.
  Clearly, $\psi_R \in \m{H}_1$.
  Using \eqref{eq:psi-grad-bd} and \eqref{eq:psi-const}, we find
  \begin{equation*}
    \int_{\Omega_R} \abs{\nab \psi_R}^2 \leq \abs{\Omega_R \setminus \Omega_{R - 1}} \And \int_{\Omega_R} \psi^2 \geq \abs{\Omega_{R - 1}}.
  \end{equation*}
  Therefore
  \begin{equation}
    \label{eq:geometry}
    \lambda(\Omega, 1) \leq \frac{\abs{\Omega_R \setminus \Omega_{R - 1}}}{\abs{\Omega_{R - 1}}} \quad \ForAll R > 1.
  \end{equation}
  Suppose for the sake of contradiction that there exists $\delta > 0$ such that
  \begin{equation*}
    \frac{\abs{\Omega_R \setminus \Omega_{R - 1}}}{\abs{\Omega_{R - 1}}} > \delta \ForAll R > 1.
  \end{equation*}
  Rearranging, we find
  \begin{equation*}
    \abs{\Omega_R} \geq (1 + \delta) \abs{\Omega_{R - 1}} \ForAll R > 1.
  \end{equation*}
  That is, the volume of $\Omega_R$ grows exponentially as $R \to \infty$, which contradicts $\Omega_R \subset B_R$.
  Therefore by Proposition~\ref{prop:lambda}\ref{item:trivial-lambda-bd} and \eqref{eq:geometry},
  \begin{equation*}
    0 \leq \lambda(\Omega, 1) \leq \inf_{R > 1} \frac{\abs{\Omega_R \setminus \Omega_{R - 1}}}{\abs{\Omega_{R - 1}}} = 0.\qedhere
  \end{equation*}
\end{proof}

\section{Existence}
\label{sec:existence}

We now consider the existence of positive bounded solutions to \eqref{eq:main}.
As noted in the introduction, this problem is linearly determined: existence can be characterized in terms of the generalized principal eigenvalue.
Throughout the section, we assume that $\Omega$ is a uniformly $\m{C}^{2,\gamma}$ domain.

We begin by proving a maximum principle for the linear problem with variable potential.
This is a consequence of recent work of Nordmann~\cite[Theorem~1]{Nordmann}, but for the reader's convenience we provide a self-contained proof in our setting.
As in Section~\ref{sec:lambda}, let $\m{L} \coloneqq \Delta + q(x)$ with $q \in \m{C}_{\mathrm{b}}^\gamma(\Omega)$.
\begin{proposition}
  \label{prop:MP}
  Let $\var \colon \partial \Omega \to [0, 1]$ be a uniformly $\m{C}^{2,\gamma}$ boundary parameter.
  Suppose $\lambda(-\m{L}, \Omega, \var) > 0.$
  If $w \in \m{C}_{\mathrm{b}}^2(\Omega) \cap \m{C}\big(\,\bar\Omega\,\big)$ satisfies
  \begin{equation}
    \label{eq:super}
    -\m{L} w \geq 0 \And \m{N}_\var w \geq 0,
  \end{equation}
  then $w \geq 0.$
\end{proposition}
The proof presented here could be easily extended to a wider class of self-adjoint operators.
However, other methods are necessary when $\m{L}$ is not self-adjoint; see \cite[Theorem~1.6(i)]{BR} for a pointwise approach.
\begin{proof}
  Define $\zeta(x) \coloneqq \op{sech}\abs{x}$ and $\zeta_\delta(x) \coloneqq \zeta(\delta x)$ for $\delta > 0$.
  Then
  \begin{equation}
    \label{eq:exp-cutoff}
    \abs{\nab \zeta_\delta(x)} = \delta \zeta_\delta(x) \abs{\tanh \abs{\delta x}} \leq \delta \zeta_\delta(x) \ForAll x \in \R^d.
  \end{equation}
  Let $w_- \coloneqq \max\{-w, 0\}$ denote the negative part of $w.$
  Recall the Hilbert space $\m{H}_\var$ from Section~\ref{sec:lambda}.
  Because $\zeta_\delta$ decays exponentially, $w$ is bounded, and $\m{N}_\var w \geq 0$, the product $\zeta_\delta w_-$ lies in $\m{H}_\var$.
  We make use of the following identity, which is easily verified through repeated integrations by parts:
  \begin{equation*}
    \int_\Omega \abs{\nab(fg)}^2 = \int_{\partial \Omega} f^2 g \partial_\nu g + \int_{\Omega} \left(g^2 \abs{\nab f}^2 - f^2 g \Delta g\right).
  \end{equation*}
  We take $f = \zeta_\delta$ and $g = w_-$.
  Using \eqref{eq:super} and \eqref{eq:exp-cutoff}, we find
  \begin{equation}
    \label{eq:energy}
    \begin{aligned}
      \int_\Omega \abs{\nab(\zeta_\delta w_-)}^2 &= \int_{\partial \Omega} \zeta_\delta^2 w_- \partial_\nu w_- +
      \int_\Omega \left(-\zeta_\delta^2 w_- \Delta w_- + w_-^2 \abs{\nab \zeta_\delta}^2\right)\\
      &\leq -\int_{\partial \Omega} \var^{-1}(1 - \var) (\zeta_\delta w_-)^2 + \int_\Omega (q + \delta^2) \abs{\zeta_\delta w_-}^2.
    \end{aligned}
  \end{equation}
  On the other hand, Proposition~\ref{prop:lambda}\ref{item:variation} states that
  \begin{equation}
    \label{eq:MP-Rayleigh}
    \int_\Omega \abs{\nab\phi}^2 \geq -\int_{\partial \Omega} \var^{-1}(1 - \var) \phi^2 + \int_\Omega \left[q + \lambda(-\m{L}, \Omega, \var)\right] \phi^2
  \end{equation}
  for all $\phi \in \m{H}_\var$.
  Setting $\phi = \zeta_\delta w_-$, \eqref{eq:energy} and \eqref{eq:MP-Rayleigh} yield
  \begin{equation*}
    \int_\Omega \left[\delta^2 - \lambda(-\m{L}, \Omega, \var)\right] \abs{\zeta_\delta w_-}^2 \geq 0.
  \end{equation*}
  Taking $\delta^2 \in \big(0, \lambda(-\m{L}, \Omega, \var)\big)$, we conclude that $w_- \equiv 0$.
  That is, $w \geq 0$.
\end{proof}
With the maximum principle in hand, we can prove our main existence result.
\begin{proof}[Proof of Theorem~\textup{\ref{thm:existence}}]
  Throughout, we fix $\const \in [0, 1]$.
  Let $f$ be positive in the sense of \ref{hyp:positive} and suppose $\lambda(\Omega, \const) < f'(0)$.
  By Proposition~\ref{prop:lambda}\ref{item:exhaustion}, there exists a ball $B$ such that
  \begin{equation*}
    \mu \coloneqq \lambda(\Omega \mid B, \const) < f'(0).
  \end{equation*}
  Let $\Omega_B$ denote the connected component of $\Omega \cap B$ whose eigenvalue is $\mu$.
  Let $\varphi$ denote the corresponding positive principal eigenfunction normalized by ${\norm{\varphi}_{L^\infty} = 1}$.
  Since $f'$ is continuous, there exists $\eps \in (0, 1)$ such that
  \begin{equation*}
    \inf_{(0, \eps]} \frac{f(s)}{s} \geq \mu.
  \end{equation*}
  It follows that $\eps \varphi$ is a subsolution of \eqref{eq:main}.
  We recall that $1$ is a supersolution.

  Let $\m{P}_t$ denote the parabolic flow given by $\m{P}_t u = v(t, \anon)$ where $v$ solves \eqref{eq:parabolic} with $v(0, \anon) = u$.
  By the comparison principle, $\m{P}_t(\eps \varphi)$ is increasing in $t$ and $\eps \varphi \leq \m{P}_t(\eps \varphi) \leq 1$.
  Thus the limit $u \coloneqq \m{P}_\infty(\eps \varphi)$ exists and satisfies $\eps \varphi \leq u \leq 1$.
  By standard parabolic estimates, $u$ solves \eqref{eq:main}.
  Moreover, the strong maximum principle implies that $u > 0$ in $\Omega$.
  So $u$ is a positive bounded solution of \eqref{eq:main}.
  \medskip

  Next, let $f$ satisfy the weak KPP hypothesis \ref{hyp:weak}.
  Suppose $\lambda(\Omega, \const) > f'(0)$ and $u$ is a nonnegative bounded solution of \eqref{eq:main}.
  Define $q \coloneqq f(u)/u$ under the convention that $q = f'(0)$ where $u = 0$.
  Let $\m{L} \coloneqq \Delta + q$.
  By the weak KPP hypothesis \ref{hyp:weak}, $q \leq f'(0)$.
  Thus the definition \eqref{eq:lambda} of $\lambda$ implies that
  \begin{equation*}
    \lambda(-\m{L}, \Omega, \const) \geq \lambda(\Omega, \const) - f'(0) > 0.
  \end{equation*}
  Because $u$ solves \eqref{eq:main}, $\m{L} u = 0$ and $\m{N}_\const u = 0$.
  By standard elliptic estimates, ${u \in \m{C}_{\mathrm{b}}^2(\Omega) \cap \m{C}\big(\bar\Omega\big)}$.
  Since $f$ is $\m{C}^{1,\gamma}$ and $f(0) = 0$, $q \in \m{C}_{\mathrm{b}}^{\gamma}(\Omega)$.
  Thus $u$ satisfies the hypotheses of the maximum principle.
  By Proposition~\ref{prop:MP}, $u = 0$.
  Thus there is no positive bounded solution of \eqref{eq:main}.
  \medskip

  We now tackle the critical case $\lambda(\Omega, \const) = f'(0)$ under additional hypotheses.
  For the remainder of the proof, we assume that $f$ is strong-KPP in the sense of \ref{hyp:strong}.
  Suppose $\Omega$ is bounded and there exists a positive bounded solution $u$ of \eqref{eq:main}.
  Then $u$ lies in the Hilbert space $\m{H}_\const$ from Section~\ref{sec:lambda}.
  We use it as a test function in the Rayleigh quotient \eqref{eq:variation}.
  By \ref{hyp:strong}, $f(u) < f'(0)u$.
  Integrating by parts, we find:
  \begin{equation}
    \label{eq:strict-quotient}
    \int_\Omega \abs{\nab u}^2 + \int_{\partial \Omega} \const^{-1}(1 - \const) u^2 = \int_{\Omega} u f(u) < f'(0) \int_\Omega u^2.
  \end{equation}
  Thus by Proposition~\ref{prop:lambda}\ref{item:variation}, $\lambda(\Omega, \const) < f'(0)$.
  In particular, there is no positive bounded solution when $\lambda(\Omega, \const) = f'(0)$.
  \medskip

  Now suppose $\Omega$ is periodic with lattice $\Lambda$ and there exists a positive bounded solution $v$ of \eqref{eq:main}.
  As noted in the introduction, $0 < v \leq 1$.
  Hence $u \coloneqq \m{P}_\infty 1$ solves \eqref{eq:main} and satisfies $0 < v \leq u \leq 1$.
  Moreover, since $\Omega$ and $1$ are periodic, so is $u$.
  That is, there exists a positive bounded \emph{periodic} solution $u$ of \eqref{eq:main}.

  Let $\Gamma \coloneqq \Omega/\Lambda$, so that $\bar\Gamma$ is a compact manifold with boundary $\partial \Gamma \coloneqq \partial \Omega/\Lambda$.
  By Theorem~1.7 and Proposition~2.5 of \cite{BR}, $\lambda(\Omega, \const)$ is the principal eigenvalue on $\Gamma$ with boundary parameter $\const$.
  (Strictly speaking, \cite{BR} only treats Dirichlet boundaries, but the Robin and Neumann proofs are identical.)
  Let $\m{H}_\const$ be $H_0^1(\Gamma)$ if $\const = 0$ and $H^1(\Gamma)$ otherwise.
  Then it is well known that the principal eigenvalue on $\Gamma$ satisfies \eqref{eq:variation} with $\Gamma$ in place of $\Omega$.
  Since $u$ is periodic, it can also be viewed as an element of $\m{H}_\const$.
  Arguing as in \eqref{eq:strict-quotient}, we again find $\lambda(\Omega, \const) < f'(0)$.
\end{proof}
\begin{remark}
  We have not used the full strength of the strong KPP hypothesis in this proof---we only required $f(u) < f'(0)u$.
  We will use the full form of \ref{hyp:strong} in the following section.
\end{remark}
\begin{remark}
  Periodic domains are akin to periodic reactions $f(x, u)$ and diffusions $\nab \cdot [A(x)\nab u]$ in the whole space.
  For an analogous existence result for periodic equations, see Theorem~2.1 of \cite{BHRoques-1}.
\end{remark}
As stated in Conjecture~\ref{conj:critical-nonexistence}, we do not expect any critical domain with a strong-KPP reaction to support a positive bounded solution.
To motivate this conjecture, we observe that we can rule out two extremes: solutions in $L^2(\Omega)$ and solutions that do not vanish locally uniformly along any subsequence.
To see the first case, we note that \eqref{eq:strict-quotient} forbids all solutions in $H^1$, and elliptic estimates ensure that $L^2$ solutions are automatically $H^1$.
For the second, one can employ a cutoff argument as in the proof of Proposition~\ref{prop:Neumann-zero} to derive a similar contradiction.

Thus the only alternative is a positive bounded solution of \eqref{eq:main} that is not square integrable but does decay along some subsequence.
We are not presently able to rule out this possibility, but it does tax the imagination.

\section{Uniqueness}
\label{sec:uniqueness}
We now come to the heart of the paper: the uniqueness of solutions to \eqref{eq:main}.
Throughout this section, we assume that $f$ is a strong-KPP reaction satisfying \ref{hyp:strong} and that $\Omega$ is a uniformly $\m{C}^{2,\gamma}$ domain with constant boundary parameter $\const \in [0, 1]$.

\subsection{Bounded domains}
Rabinowitz~\cite{Rabinowitz} has previously proven uniqueness on bounded domains.
We recall the proof for the reader's convenience.
It serves as a blueprint for our arguments in the unbounded case.
\begin{theorem}[\cite{Rabinowitz}]
  \label{thm:bounded}
  If $\Omega$ is bounded and $\const = 0$, then \eqref{eq:main} admits at most one solution.
\end{theorem}
As noted in the introduction, Rabinowitz in fact treats more general self-adjoint elliptic operators, and the first author has shown that self-adjointness is not necessary on bounded domains~\cite{Berestycki}.
\begin{proof}
  We are free to assume there exists a positive bounded solution.
  Let $u_1$ and $u_2$ be two such solutions (not necessarily distinct).
  By the strong maximum principle and the Hopf lemma, $u_1$ and $u_2$ are comparable in the sense that
  \begin{equation}
    \label{eq:comparable}
    C^{-1}u_2 \leq u_1 \leq C u_2 \quad \text{for some } C \in [1, \infty).
  \end{equation}
  Define
  \begin{equation*}
    \ubar{\mu} \coloneqq \inf \left\{\mu > 0 \mid u_1 \leq \mu u_2 \enspace \text{in } \Omega\right\},
  \end{equation*}
  which is finite by \eqref{eq:comparable}.

  We claim that $\ubar{\mu} \leq 1$.
  Suppose for the sake of contradiction that $\ubar{\mu} > 1$.
  Let $w \coloneqq \ubar{\mu} u_2 - u_1 \geq 0$.
  By the strong KPP property and $u_2 > 0$,
  \begin{equation*}
    -\Delta w = \ubar{\mu}f(u_2) - f(u_1) > f(\ubar{\mu}u_2) - f(u_1) = q w
  \end{equation*}
  for some $q \in L^\infty(\Omega)$.
  By the strong maximum principle, $w > 0$.
  Again, the strong maximum principle and the Hopf lemma imply that $w$ is comparable to $u_2$.
  Thus there exists $\eps > 0$ such that $w \geq \eps u_2$.
  That is, $u_1 = \ubar{\mu}u_2 - w \leq (\ubar{\mu} - \eps)u_2$, contradicting the definition of $\ubar{\mu}$.
  So in fact $u_1 \leq u_2$, and $u_1 = u_2$ by symmetry.
\end{proof}
This proof relies on the Hopf lemma, and thus on the interior ball property.
To our knowledge, uniqueness in less regular domains is open.
Even if $\Omega$ is somewhat irregular, it seems plausible that all positive bounded solutions of fairly general elliptic problems on $\Omega$ are comparable to one another.
Indeed, any such solution is likely comparable to the generalized principle eigenfunction constructed in \cite{BNV}.
If this is case, the above proof of uniqueness can be extended to irregular domains.
We expect this to hold, conservatively, in the case of Lipschitz boundary.
\begin{conjecture}
  If $f$ is strong-KPP, $\Omega$ is a bounded Lipschitz domain, and $\const = 0$, then \eqref{eq:main} admits at most one solution.
\end{conjecture}

\subsection{Domain decomposition}
The proof of Theorem~\ref{thm:bounded} points to certain difficulties in unbounded domains.
When $\Omega$ is unbounded, positive solutions may vanish at infinity.
If distinct solutions vanish at different rates, they need not be comparable in the sense of \eqref{eq:comparable}.
It is reasonable to expect all solutions to vanish at comparable rates, but it is not clear how to prove such a fact.

We sidestep this issue by decomposing $\Omega$ into two parts, $\Omega_-$ and $\Omega_+$.
Roughly speaking, $\Omega_-$ is the region on which solutions cannot vanish at infinity.
Solutions do vanish at infinity in the remainder $\Omega_+$, but it is sufficiently ``narrow'' that we can apply our maximum principle.
This additional tool substitutes for precise knowledge about rates of decay and suffices to show uniqueness.

To capture these properties, we rely on the principal limit spectrum $\Sigma$ defined in the introduction.
In the following, $\Omega_\pm$ is a uniformly $\m{C}^{2,\gamma}$ open set that need not be connected (and may be empty) and $\var_\pm$ is a uniformly $\m{C}^{2,\gamma}$ boundary parameter on $\Omega_\pm$.
These notions of regularity are made precise in Definition~\ref{def:smooth}.
\begin{definition}
  \label{def:ample-narrow}
  Fix $\mu > 0$.
  A pair $(\Omega_-, \var_-)$ is $\mu$-\emph{ample} if $\Sigma(\Omega_-, \var_-) \subset [0, \mu)$.
  A pair $(\Omega_+, \var_+)$ is $\mu$-\emph{narrow} if $\Sigma(\Omega_+, \var_+) \subset (\mu, \infty)$ or, equivalently, $\lambda(\Omega_+, \var_+) > \mu$.
  We say $(\Omega_-, \Omega_+, \var_-, \var_+)$ is a $\mu$-\emph{ample--narrow decomposition} of $(\Omega, \const)$ if the following hold: $(\Omega_\pm, \var_\pm) \preceq (\Omega, \const)$, $\Omega = \Omega_- \cup \Omega_+$, $(\Omega_-, \var_-)$ is $\mu$-ample, $(\Omega_+, \var_+)$ is $\mu$-narrow, and $\max\{\var_+, \var_-\} = \const$ on $\partial \Omega$, where we extend $\var_\pm$ by $0$ to $\partial \Omega$.
\end{definition}
Conceptually, $\lambda(\Omega, \var) > \mu$ if the absorption from the $\var$-boundary is stronger than rate-$\mu$ growth in the interior.
This means that every point in $\Omega$ is relatively close to $\partial \Omega$; that is, the domain is ``narrow.''
Conversely, if $\lambda(\Omega^*, \var^*) < \mu$ for every connected limit $(\Omega^*, \var^*)$, then $(\Omega, \var)$ must be sufficiently capacious even at infinity.
We express this notion through the word ``ample.''

We wish to prove uniqueness on domains $\Omega$ that are \emph{noncritical} in the sense that $f'(0) \not \in \bar{\Sigma}(\Omega, \const)$.
However, this condition on limit eigenvalues is difficult to use directly.
Instead, we prove a geometric result relating this condition to ample and narrow sets.
\begin{theorem}
  \label{thm:domain-decomposition}
  Fix $\mu > 0.$
  Then $(\Omega, \const)$ admits a $\mu$-ample--narrow decomposition if and only if $\mu \not \in \bar{\Sigma}(\Omega, \const)$.
\end{theorem}
\noindent
The ample--narrow decomposition is far from unique.
Roughly speaking, we can augment $\Omega_-$ and curtail $\Omega_+$ by any bounded region and still preserve the decomposition.

We will make use of one technical operation.
Given open $U \subset \Omega$, we define the nested sets
\begin{equation}
  \label{eq:nest}
  U_r \coloneqq \big\{x \in U \mid \op{dist}_\Omega(x, \Omega \setminus U) > r\big\} \quad \text{for } r > 0.
\end{equation}
We wish to ``cut off'' $\Omega$ so that it lies within $U$ but contains $U_1$.
\begin{definition}
  \label{def:truncation}
  Let $(\Omega, \var)$ be a uniformly $\m{C}^{2,\gamma}$ domain and let $U \subset \Omega$ be open.
  Then $(\Omega_U, \var_U)$ is a \emph{truncation} of $(\Omega, \var)$ in $U$ if it is a uniformly $\m{C}^{2,\gamma}$ open set such that $U_1 \subset \Omega_U \subset U$, $(\Omega_U, \var_U) \preceq (\Omega, \var)$, and $\var_U = \var$ on $\partial \Omega \cap \bar{U}_1$.
\end{definition}
\begin{lemma}
  \label{lem:truncation}
  If $(\Omega, \var)$ is uniformly $\m{C}^{2,\gamma}$ and $U \subset \Omega$ is open, there exists a truncation of $(\Omega, \var)$ in $U$ in the sense of Definition~\textup{\ref{def:truncation}}.
\end{lemma}
This lemma comes as no surprise, but for lack of a clean reference, we provide a proof.
\begin{proof}
  We first construct a uniformly smooth set $\Omega_U \subset \Omega$ such that
  \begin{equation}
    \label{eq:strip}
    U_{3/4} \subset \Omega_U \subset U_{1/4}.
  \end{equation}
  Define $f(x) \coloneqq \op{dist}_\Omega(x, \Omega \setminus U)$ and $g(x) \coloneqq \min\{\max\{f(x), \, 1/3\}, \, 2/3\}$.
  Since $\Omega$ is uniformly smooth,
  \begin{equation*}
    \op{dist}_{\R^d}\big(f^{-1}(1/3, 2/3), f^{-1}(1/4, 3/4)^c\big) > 0.
  \end{equation*}
  It follows that there is a mollification $h$ of $g$ such that
  \begin{equation}
    \label{eq:positive-dist}
    \op{dist}_\Omega\big(h^{-1}(1/3, 2/3), f^{-1}(1/4, 3/4)^c\big) > 0.
  \end{equation}

  Sard's theorem provides $c \in (1/3, 2/3)$ such that $h^{-1}(c)$ is a (locally) smooth manifold, perhaps with corners where the level set meets $\partial \Omega$.
  Let ${V \coloneqq h^{-1}(c, 1]}$, which satisfies $U_{3/4} \subset V \subset U_{1/4}$.
  In fact, \eqref{eq:positive-dist} implies that $\partial V$ is uniformly separated from both $U_{1/4}$ and $U_{3/4}^c$.
  We therefore have space to uniformly smooth $\partial V$ and make $\partial V$ smoothly join $\partial \Omega$ while still remaining in the strip $U_{1/4} \setminus \bar{U}_{3/4}$.
  Let $\Omega_U$ be this uniformly smooth modification, which satisfies \eqref{eq:strip}.

  Next, let $Z \coloneqq \bar{U}_{3/4} \setminus (\bar{\Omega \setminus U_{3/4}})$ and $W \coloneqq (\Omega \setminus U_1) \setminus \bar U_1$.
  Then $\{Z, W\}$ is an open cover of $\bar\Omega$ viewed as a manifold with boundary.
  Let $\{\psi, \phi\}$ be a smooth partition of unity on $\bar\Omega$ adapted to this cover.
  Because $\op{dist}_\Omega(Z, \bar\Omega \setminus W) > 0$, we can arrange $\psi$ (and $\phi$) to be uniformly $\m{C}^{2,\gamma}$.

  Finally, let $\var_U \coloneqq \psi \var \tbf{1}_{\partial \Omega}$ on $\partial \Omega_U$.
  Because $\psi$ is uniformly smooth and supported in $U_{3/4}$, $\var_U$ is uniformly smooth and $(\Omega_U, \var_U) \preceq (\Omega, \var)$.
  Moreover, $\psi \equiv 1$ on $U_1$, so $\var_U = \var$ on $\partial \Omega \cap \bar{U}_1$.
  Hence $(\Omega_U, \var_U)$ is a truncation of $(\Omega, \var)$ in $U$.
\end{proof}
We now prove our main geometric result.
\begin{proof}[Proof of Proposition~\textup{\ref{thm:domain-decomposition}}]
  One direction is straightforward.
  Suppose $(\Omega, \const)$ admits a $\mu$-ample--narrow decomposition $(\Omega_-, \Omega_+, \var_-, \var_+)$.
  Let $\Omega^*$ be a connected limit of $\Omega$ along a sequence $(x_n)_{n \in \N} \subset \Omega$, which necessarily inherits the boundary parameter $\const$.
  Then $\Omega^*$ is a connected component of the limit of the translates $\Omega - x_n$.
  Recall that $\op{dist}_\Omega$ denotes the distance induced on the metric space $\Omega$ through its inclusion in $\R^d$.
  If $\op{dist}_\Omega(x_n, \Omega_-) \to \infty$ as $n \to \infty$, then $(\Omega^*, \const)$ is a connected limit of $(\Omega_+, \var_+)$ alone.
  Recall that $(\Omega_+, \var_+)$ is $\mu$-narrow.
  By \eqref{eq:identity}, Lemma~\ref{lem:semicontinuous}, and Definition~\ref{def:ample-narrow},
  \begin{equation*}
    \lambda(\Omega^*, \const) \geq \lambda(\Omega_+, \var_+) > \mu.
  \end{equation*}

  Next, suppose a subsequence of $\op{dist}_\Omega(x_n, \Omega_-)$ remains bounded as $n \to \infty$.
  Restricting to a further subsequence, we can extract a connected limit $(\Omega_-^*, \var_-^*)$ of $(\Omega_-, \var_-)$.
  It is straightforward to check that the partial order $\preceq$ is preserved under connected limits, so by Definition~\ref{def:ample-narrow}, $(\Omega_-^*, \var_-^*) \preceq (\Omega^*, \const)$.
  By Proposition~\ref{prop:lambda}\ref{item:monotone},
  \begin{equation*}
    \lambda(\Omega^*, \const) \leq \lambda(\Omega_-^*, \var_-^*) \leq \sup \Sigma(\Omega_-, \var_-).
  \end{equation*}
  Because $(\Omega_-, \var_-)$ is $\mu$-ample, $\sup \Sigma(\Omega_-, \var_-) \leq \mu$.
  We claim that the inequality is strict.
  To see this, suppose otherwise.
  Then there exists a sequence of connected limits $(\Omega_-^n, \var_-^n)$ of $(\Omega_-, \var_-)$ such that
  \begin{equation}
    \label{eq:ample-approach}
    \lambda(\Omega_-^n, \var_-^n) \nearrow \mu \quad \text{as } n \to \infty.
  \end{equation}
  For each $n \in \N$, $(\Omega_-^n, \var_-^n)$ is constructed from a sequence of translates
  \begin{equation*}
    (\tau_{x_{n,m}}\Omega_-, \tau_{x_{n,m}}\var_-)_{m \in \N}.
  \end{equation*}
  Restricting to a subsequence of $n$, we can assume that the sequence $(\Omega_-^n)_{n \in \N}$ has its own connected limit $(\Omega_-^\infty, \var_-^\infty)$.
  By Lemma~\ref{lem:semicontinuous} and \eqref{eq:ample-approach},
  \begin{equation}
    \label{eq:critical-contradiction}
    \lambda(\Omega_-^\infty, \var_-^\infty) \geq \mu.
  \end{equation}
  However, we can diagonalize along the double sequence $(\tau_{x_{n,m}}\Omega_-, \tau_{x_{n,m}}\var_-)_{n,m \in \N}$ to see that $(\Omega_-^\infty, \var_-^\infty)$ is itself a connected limit of $(\Omega_-, \var_-)$.
  Thus \eqref{eq:critical-contradiction} contradicts the $\mu$-ampleness of $(\Omega_-, \var_-)$.
  It follows that $\sup \Sigma(\Omega_-, \var_-) < \mu$, as claimed.

  If $\Omega^*$ is a connected limit of $\Omega$, we have now shown that one of two alternatives holds:
  \begin{equation*}
    \lambda(\Omega^*, \const) \geq \lambda(\Omega_+, \var_+) > \mu
  \end{equation*}
  or
  \begin{equation*}
    \lambda(\Omega^*, \const) \leq \sup \Sigma(\Omega_-, \var_-) < \mu.
  \end{equation*}
  It follows that $\mu \not \in \bar{\Sigma}(\Omega, \var)$.
  \medskip

  We now tackle the reverse direction.
  Assume $\mu \not \in \bar{\Sigma}(\Omega, \const)$, so there exists $\delta > 0$ such that
  \begin{equation}
    \label{eq:gap}
    \Sigma(\Omega, \const) \cap (\mu - 2 \delta, \mu + 2 \delta) = \emptyset.
  \end{equation}
  We wish to show that $(\Omega, \const)$ admits a $\mu$-ample--narrow decomposition.
  We first construct the narrow part $(\Omega_+, \var_+)$.
  Fix $R > 0$ such that
  \begin{equation}
    \label{eq:large-radius}
    \lambda(B_1, 0) R^{-2} < \delta.
  \end{equation}
  In the following, let $\op{comp}_xU$ denote the connected component of a domain $U$ containing $x$.
  Let $\var \coloneqq \const \tbf{1}_{\partial \Omega}$.
  We use the boundary parameter $\var$ on various subsets of $\Omega$ via restriction.
  Define
  \begin{equation*}
    \Omega_+^0 \coloneqq \big\{x \in \Omega \mid \lambda\big(\op{comp}_x(\Omega \cap B_{2R}(x)), \var\big) > \mu + \delta\big\}.
  \end{equation*}
  This is our ``first pass'' at the narrow part.
  It is not difficult to verify that the map
  \begin{equation*}
    x \mapsto \lambda\big(\op{comp}_x(\Omega \cap B_{2R}(x)), \var\big)
  \end{equation*}
  is continuous, so $\Omega_+^0$ is open.

  Fix $y \in \R^d$ and suppose $B_R(y)$ intersects $\Omega_+^0$.
  Take $x \in \Omega_+^0 \cap B_R(y)$.
  Because $\abs{x - y} < R$, we have $B_R(y) \subset B_{2R}(x)$.
  It follows that
  \begin{equation*}
    \op{comp}_x(\Omega_+^0 \cap B_R(y)) \subset \op{comp}_x(\Omega \cap B_{2R}(x)).
  \end{equation*}
  By Proposition~\ref{prop:lambda}\ref{item:monotone} and the definition of $\Omega_+^0$,
  \begin{equation*}
    \lambda\big(\op{comp}_x(\Omega_+^0 \cap B_R(y)), \var\big) \geq \lambda\big(\op{comp}_x(\Omega \cap B_{2R}(x)), \var\big) > \mu + \delta.
  \end{equation*}
  Now, $x \in \Omega_+^0 \cap B_R(y)$ was arbitrary, so every connected component of $\Omega_+^0 \cap B_R(y)$ has principal eigenvalue at least $\mu + \delta$.
  Therefore
  \begin{equation}
    \label{eq:strictly-narrow}
    \lambda\big(\Omega_+^0 \cap B_R(y), \var\big) = \lambda\big(\Omega_+^0 \mid B_R(y), \var\big) \geq \mu + \delta.
  \end{equation}
  Since this holds for all $y \in \R^d$, Theorem~\ref{thm:Lieb} and our choice \eqref{eq:large-radius} morally imply that
  \begin{equation*}
    \lambda(\Omega_+^0, \var) > \mu.
  \end{equation*}
  This is the desired inequality for $\mu$-narrowness.

  There is a catch, however.
  Generally, $\Omega_+^0$ has corners and $\var$ is discontinuous.
  To circumvent this, we let $(\Omega_+, \var_+)$ be a truncation of $(\Omega, \const)$ in $\Omega_+^0$, as provided by Lemma~\ref{lem:truncation}.
  Using Proposition~\ref{prop:lambda}\ref{item:monotone} and \eqref{eq:strictly-narrow}, we see that
  \begin{equation*}
    \lambda\big(\Omega_+ \mid B_R(y), \var_+\big) \geq \mu + \delta
  \end{equation*}
  for all $y \in \R^d$.
  Thus by Theorem~\ref{thm:Lieb} and \eqref{eq:large-radius},
  \begin{equation*}
    \lambda(\Omega_+, \var_+) > \mu.
  \end{equation*}
  That is, $(\Omega_+, \var_+)$ is $\mu$-narrow.

  Next, we construct the ample part.
  In the following, let $G \coloneqq \Omega \setminus \Omega_+^0$.
  We claim that there exists $K \geq 2$ such that
  \begin{equation}
    \label{eq:uniform-negative}
    \lambda\big(\op{comp}_x(\Omega \cap B_K(x)), \var\big) \leq \mu-\delta \ForAll x \in G.
  \end{equation}
  To see this, suppose otherwise.
  Then there exists a sequence $(x_n)_{n \in \N} \subset G$ such that
  \begin{equation}
    \label{eq:lower}
    \lambda\big(\op{comp}_{x_n}(\Omega \cap B_n(x_n)), \var\big) > \mu-\delta
  \end{equation}
  for all $n \in \N$.
  On the other hand, because $x_n \not \in \Omega_+^0$, we have
  \begin{equation}
    \label{eq:not-narrow}
    \lambda\big(\op{comp}_{x_n}(\Omega \cap B_{2R}(x_n)), \var\big) \leq \mu+\delta.
  \end{equation}
  Extracting a subsequence, we can assume that $(\tau_{x_n} \Omega)_{n \in \N}$ has a connected limit $\Omega^*$.
  Let $\var^* \coloneqq \const\tbf{1}_{\partial \Omega^*}$.
  Morally, Lemma~\ref{lem:semicontinuous} and \eqref{eq:lower} imply that
  \begin{equation*}
    \lambda(\Omega^*, \const) \geq \mu-\delta.
  \end{equation*}
  Strictly speaking, we cannot apply Lemma~\ref{lem:semicontinuous} because $\op{comp}_{x_n}(\Omega \cap B_n(x_n))$ is generally not smooth and $\var$ is generally discontinuous.
  Again, we can bypass these difficulties by dealing with a truncation $(\ti{\Omega}_n, \ti{\var}_n)$ of $(\Omega, \const)$ in $\op{comp}_{x_n}(\Omega \cap B_n(x_n))$.
  By Proposition~\ref{prop:lambda}\ref{item:monotone} and \eqref{eq:lower},
  \begin{equation*}
    \lambda(\ti \Omega_n, \ti \var_n) \geq \mu-\delta.
  \end{equation*}
  Moreover, $(\ti{\Omega}_n, \ti{\var}_n)$ coincides with $(\op{comp}_{x_n}(\Omega \cap B_n(x_n)), \sigma)$ within $B_{n-1}(x_n)$, so the sequence $(\tau_{x_n}\ti \Omega_n, \tau_{x_n}\ti \var_n)_n$ also has connected limit $(\Omega^*, \const)$.
  Therefore Lemma~\ref{lem:semicontinuous} \emph{does} imply that
  \begin{equation}
    \label{eq:lower-contradiction}
    \lambda(\Omega^*, \const) \geq \mu-\delta.
  \end{equation}

  On the other hand, locally uniform $\m{C}^{2, \al}$ convergence implies that
  \begin{equation*}
    (\op{comp}_0(\tau_{x_n}\Omega \cap B_{2R}), \tau_{x_n}\var) \to (\op{comp}_0(\Omega^* \cap B_{2R}), \var^*).
  \end{equation*}
  As noted in Remark~\ref{rem:continuous}, the principal eigenvalue is continuous on uniformly bounded domains.
  Thus Proposition~\ref{prop:lambda}\ref{item:monotone} and \eqref{eq:not-narrow} yield
  \begin{equation}
    \label{eq:lower-semicont}
    \begin{aligned}
      \lambda(\Omega^*, \const) &\leq \lambda\big(\op{comp}_0(\Omega^* \cap B_{2R}), \var^*\big)\\
      &\hspace{1.5cm}= \lim_{n \to \infty} \lambda\big(\op{comp}_0(\tau_{x_n}\Omega \cap B_{2R}), \tau_{x_n}\var\big) \leq \mu + \delta.
    \end{aligned}
  \end{equation}
  In combination with \eqref{eq:lower-contradiction}, we see that $\Omega_*$ is a connected limit of $\Omega$ with
  \begin{equation*}
    \lambda(\Omega^*, \const) \in [\mu-\delta, \mu+\delta].
  \end{equation*}
  This contradicts \eqref{eq:gap}.
  Therefore \eqref{eq:uniform-negative} does hold for some $K \geq 2$.

  Next, define
  \begin{equation*}
    \Omega_-^0 \coloneqq \bigcup_{x \in G} \op{comp}_x(\Omega \cap B_K(x)).
  \end{equation*}
  We note that $G \subset \Omega_-^0$.
  Moreover, because $K \geq 2$,
  \begin{equation}
    \label{eq:union}
    \Omega_-^0 \cup \Omega_+ = \Omega.
  \end{equation}

  Suppose $x \in \Omega_-^0$, so that ${x \in \op{comp}_{x_0}(\Omega \cap B_K(x_0))}$ for some $x_0 \in G$.
  It follows that $B_K(x_0) \subset B_{2K}(x)$ and
  \begin{equation*}
    \op{comp}_{x_0}(\Omega \cap B_K(x_0)) \subset \op{comp}_{x}(\Omega_-^0 \cap B_{2K}(x)).
  \end{equation*}
  By Proposition~\ref{prop:lambda}\ref{item:monotone} and \eqref{eq:uniform-negative}, we have
  \begin{equation}
    \label{eq:ample-ball}
    \lambda\big(\op{comp}_{x}(\Omega_-^0 \cap B_{2K}(x)), \var\big) \leq \lambda\big(\op{comp}_{x_0}(\Omega \cap B_K(x_0)), \var\big) \leq \mu-\delta.
  \end{equation}

  We now smooth $\Omega_-^0$ into $\Omega_-$.
  Define $U \coloneqq \{x \in \Omega \mid \op{dist}_\Omega(x, \Omega_-^0) < 1\}$ and let $(\Omega_-, \var_-)$ be a truncation of $(\Omega, \const)$ in $U$.
  Because $\Omega_-^0 \subset U_1$ in the notation of \eqref{eq:nest}, we have $(\Omega_-^0, \sigma) \preceq (\Omega_-, \var_-)$.
  Now take $x \in \Omega_-$.
  Since $x \in U$, $\op{dist}_\Omega(x, \Omega_-^0) < 1$.
  Hence there exists $x_0 \in \Omega_-^0$ such that
  \begin{equation*}
    \op{comp}_{x_0}(\Omega_-^0 \cap B_{2K}(x_0)) \subset \op{comp}_{x}(\Omega_- \cap B_{2K+1}(x)).
  \end{equation*}
  By Proposition~\ref{prop:lambda}\ref{item:monotone} and \eqref{eq:ample-ball},
  \begin{equation*}
    \op{comp}_{x}(\Omega_- \cap B_{2K+1}(x)) \leq \mu- \delta \ForAll x \in \Omega_-.
  \end{equation*}
  Mimicking the reasoning leading to \eqref{eq:lower-semicont}, we see that every connected limit $(\Omega_-^*, \var_-^*)$ of $(\Omega_-, \var_-)$ satisfies
  \begin{equation*}
    \lambda(\Omega_-^*, \var_-^*) \leq \mu - \delta.
  \end{equation*}
  Therefore $(\Omega_-, \var_-)$ is $\mu$-ample.

  Since $\Omega_-^0 \subset \Omega_-$, \eqref{eq:union} implies that $\Omega = \Omega_- \cup \Omega_+$.
  Similarly, $K \geq 2$ and our truncation constructions of $\var_\pm$ ensure that $\max\{\var_+,\var_-\} = \const$.
  Therefore $(\Omega_-, \Omega_+, \var_-, \var_+)$ is a $\mu$-ample--narrow decomposition of $(\Omega, \const)$.
\end{proof}

\subsection{Uniqueness}
We are nearly in a position to prove Theorem~\textup{\ref{thm:uniqueness}}.
Our final auxiliary result constrains the vanishing of solutions at infinity.
\begin{lemma}
  \label{lem:zero-limit}
  Let $\Omega^*$ be a connected limit of $\Omega$ along $(x_n)_n \subset \Omega$.
  Let $u$ (resp. $v$) be a positive $\m{C}_{\mathrm{b}}^{2, \gamma}$ supersolution (resp. subsolution) of \eqref{eq:main} on $(\Omega, \const)$.
  If $\lambda(\Omega^*, \const) < f'(0),$ then $(\tau_{x_n}u)_n$ does not vanish locally uniformly along any subsequence.
  Conversely, if $\lambda(\Omega^*, \const) > f'(0),$ then $\tau_{x_n}v \to  0$ locally uniformly as $n \to \infty$.
\end{lemma}
\begin{proof}
  We prove the second part first.
  By elliptic regularity, every subsequential limit $v^*$ of $\tau_{x_n}v$ is a nonnegative subsolution of \eqref{eq:main} on $(\Omega^*, \const)$.
  Let $\m{P}_t^*$ denote the parabolic flow on $(\Omega^*, \const)$ by analogy with \eqref{eq:parabolic}.
  Because $v^*$ is a subsolution, $\m{P}_\infty^* v^* \geq v^*$ is a bounded nonnegative solution of \eqref{eq:main} on $(\Omega^*, \const)$.
  By Theorem~\ref{thm:existence}, $\m{P}_\infty^* v^* = 0$ if $\lambda(\Omega^*, \const) > f'(0)$.
  It follows that $v^* = 0$.
  Because the subsequential limit is unique, $\tau_{x_n}v \to 0$ along the entire sequence.

  On the other hand, suppose a supersolution $u$ vanishes along $(x_n)_n$.
  Define
  \begin{equation*}
    w_n \coloneqq \frac{\tau_{x_n}u}{u(x_n)} > 0,
  \end{equation*}
  so that $w_n(0) = 1$.
  Then $(w_n)_n$ has a locally uniform subsequential limit $w^*$ that satisfies
  \begin{equation*}
    \begin{cases}
      -\Delta w^* \geq f'(0)w^* & \text{in } \Omega^*,\\
      \m{N}_\const w^* \geq 0 & \text{on } \partial \Omega^*,\\
      w^*(0) = 1.
    \end{cases}
  \end{equation*}
  By the strong maximum principle, $w^* > 0$ in $\Omega^*$.
  Using $w^*$ in \eqref{eq:lambda}, we see that $\lambda(\Omega^*, \const) \geq f'(0)$.
  Taking the contrapositive, we obtain Lemma~\ref{lem:zero-limit}.
\end{proof}
We can now prove our main uniqueness result.
\begin{proof}[Proof of Theorem~\textup{\ref{thm:uniqueness}}]
  Suppose $f'(0) \not \in \bar{\Sigma}(\Omega, \const)$.
  Then by Proposition~\ref{thm:domain-decomposition}, $(\Omega, \const)$ admits an $f'(0)$-ample--narrow decomposition $(\Omega_-, \Omega_+, \var_-, \var_+)$ in the sense of Definition~\ref{def:ample-narrow}.

  If $\lambda(\Omega, \const) > f'(0)$, then Theorem~\ref{thm:existence} states that \eqref{eq:main} admits no positive bounded solution.
  We cannot have $\lambda(\Omega, \const) = f'(0)$ because $\lambda(\Omega, \const) \in \Sigma(\Omega, \const)$.
  Thus it suffices to treat the case $\lambda(\Omega, \const) < f'(0)$, which implies that $\Omega_-$ is nonempty.
  Moreover, Theorem~\ref{thm:existence} ensures the existence of a positive bounded solution of \eqref{eq:main}.
  Let $u_1, u_2 > 0$ be two such solutions (not necessarily distinct).
  
  First, we claim there exists $\mu < \infty$ such that $u_1 \leq \mu u_2$ in $\Omega_-$.
  For the sake of contradiction, suppose instead that $\inf_{\Omega_-}(u_2 - u_1/n) < 0$ for all $n \in \N$.
  Then for each $n \in \N$, there exists $x_n \in \Omega_{-}$ such that
  \begin{equation}
    \label{eq:negative}
    u_2(x_n) - u_1(x_n)/n < 0.
  \end{equation}
  Because $\var_- \leq \const$ and $u_2 \geq 0$, we have
  \begin{equation*}
    \m{N}_{\var_-} u_2 = \frac{\var_-}{\const} \m{N}_\const u_2 + \left(1 - \frac{\var_-}{\const}\right) u_2 \geq 0.
  \end{equation*}
  So $u_2$ is a supersolution of \eqref{eq:main} on $(\Omega_-, \var_-)$.

  We take $n \to \infty$; for simplicity, we suppress subsequential notation.
  There exist subsequential connected limits $\Omega_-^* \subset \Omega^*$ of $\Omega_- \subset \Omega$ along the sequence $(x_n)_n$.
  Also, there exists a limit $u_2^*$ of $(\tau_{x_n}u_2)_n$ solving \eqref{eq:main} on $(\Omega^*, \const)$.
  By ampleness, Lemma~\ref{lem:zero-limit}, and the strong maximum principle, $u_2^* > 0$ in $\Omega_-^*$.
  Since $\Omega^*$ is connected, $u_2^* > 0$ in $\Omega^*$.
  On the other hand, \eqref{eq:negative} and the boundedness of $u_1$ imply that $u_2^*(0) = 0$.
  Therefore $0 \in \partial \Omega^*$.
  This contradicts the boundary condition $\m{N}_\const u_2^* = 0$ and the Hopf lemma unless $\const = 0$.
  So, for the time being, assume $\const = 0.$
  Let $\nu^*$ denote the outward normal vector field on $\partial \Omega^*$.
  We claim that $\partial_{\nu^*} u_2^*(0) = 0$, which again contradicts the Hopf lemma.

  To see this, suppose for the sake of contradiction that $m \coloneqq -\partial_{\nu^*} u_2^*(0) > 0$.
  Recall that $\nu$ denotes the outward normal on the original boundary $\partial\Omega$.
  By the collar neighborhood theorem and the uniform regularity of $\Omega$, there exists $\delta > 0$ such that $\nu$ has a natural smooth extension to the $\delta$-neighborhood $\partial^\delta \Omega$ of the boundary.
  Precisely, the vector field $\nu$ identifies $\partial^\delta \Omega$ with the product manifold $\partial \Omega \times [0, \delta)$.
  We denote product coordinates by $x = (\bar{x}, h)$, so that the ``height'' $h$ denotes distance from $\partial \Omega$ and $\pa{}{h} = \nu$.

  Since $0 \in \partial\Omega^*$, $d(x_n, \partial \Omega) \to 0$ as $n \to \infty$.
  In particular, there exists $N_0 \in \N$ such that $x_n \in \partial^\delta \Omega$ for all $n \geq N_0$.
  When $n \geq N_0$, the derivative $\partial_\nu u_2(x_n)$ is well-defined.
  Elliptic regularity implies that
  \begin{equation*}
    -\partial_\nu u_2(x_n) \to -\partial_{\nu^*} u_2^*(0) = m > 0 \quad \text{as } n \to \infty.
  \end{equation*}
  Moreover, $\m{C}^{1,\al}$ elliptic estimates provide $\eps > 0$ and $N_1 \geq N_0$ such that for all $n \geq N_1$,
  \begin{equation}
    \label{eq:normal-lower}
    -\partial_\nu u_2(x_n) \geq m/2 \quad \text{on } \bar{B}_\eps(\bar{x}_n) \times (0, \eps).
  \end{equation}
  Here $\bar{x}_n$ denotes the projection of $x_n$ to $\partial\Omega$ and $\bar{B}_\eps(\bar{x}_n)$ denotes its $\eps$-neighborhood in $\partial \Omega$ with Riemannian metric induced from $\R^d$.
  The height $h_n$ of $x_n$ tends to $0$, so we may assume that $h_n < \eps$ for all $n \geq N_1$.

  On the other hand, uniform $\m{C}^1$-regularity implies that $M \coloneqq \sup_\Omega \abs{\nab u_1} < \infty$.
  Integrating from the boundary, \eqref{eq:normal-lower} yields
  \begin{equation}
    \label{eq:linear-bounds}
    u_2(\bar{x}, h) \geq \frac{m}{2}h \And u_1(\bar{x}, h) \leq Mh \quad \text{on } \bar{B}_\eps(\bar{x}_n) \times (0, \eps)
  \end{equation}
  for all $n \geq N_1$.
  Since $x_n$ lies in the neighborhood, \eqref{eq:linear-bounds} contradicts \eqref{eq:negative} when $n > 2M/m$.

  We now return to the general setting $\const \in [0, 1]$.
  Together, these contradictions imply the existence of $\mu < \infty$ such that $u_1 \leq \mu u_2$ in $\Omega_-$.
  Therefore
  \begin{equation*}
    \ubar{\mu} \coloneqq \inf \left\{\mu > 0 \mid u_1 \leq \mu u_2 \enspace \text{in } \Omega_{-}\right\}
  \end{equation*}
  is finite.
  By continuity, $w \coloneqq \ubar{\mu}u_2 - u_1 \geq 0$ in $\Omega_{-}$.
  We claim that $\ubar{\mu} \leq 1$.
  Suppose otherwise for the sake of contradiction.

  Consider $w$ on $\Omega_+$.
  The strong KPP property \ref{hyp:strong} implies that
  \begin{equation*}
    -\Delta w = \ubar{\mu}f(u_2) - f(u_1) > f(\ubar\mu u_2) - f(u_1) \eqqcolon q w.
  \end{equation*}
  Setting $\m{L} \coloneqq \Delta + q$, this reads $-\m{L} w > 0$.
  Now $u_i \in \m{C}_{\mathrm{b}}^{2, \gamma}$ and $f \in \m{C}^{1, \gamma}$, so the difference quotient $q$ is uniformly $\m{C}^\gamma$.
  Moreover, \ref{hyp:strong} yields $q \leq f'(0)$.

  Both $u_1$ and $u_2$ satisfy the $\const$-boundary condition on $\partial \Omega$, so $\m{N}_\const w = 0$ there.
  Now $(\Omega_+, \var_+) \preceq (\Omega, \const)$, so $\var_+ = 0$ on $\partial \Omega_+ \cap \Omega$ and $\var_+ \leq \const$ on $\partial \Omega_+ \cap \partial \Omega$.
  Because $\Omega = \Omega_- \cup \Omega_+$, we have $\partial \Omega_+ \cap \Omega \subset \Omega_-$.
  Hence $w \geq 0$ there, and it follows that $\m{N}_{\var_+} w = w \geq 0$ on $\partial \Omega_+ \cap \Omega$.
  Next, Definition~\ref{def:ample-narrow} ensures that $\var_+ = \const$ on $\partial \Omega_+ \setminus \partial \Omega_-$.
  There, we have $\m{N}_{\var_+} w = \m{N}_\const w = 0$.
  Finally, on $\partial \Omega_+ \cap \partial \Omega_-$, we have $w \geq 0$ and $\var_+ \leq \const$, so
  \begin{equation*}
    \m{N}_{\var_+} w = \frac{\var_+}{\const} \m{N}_\const w + \left(1 - \frac{\var_+}{\const}\right) w \geq 0.
  \end{equation*}
  We have thus shown that
  \begin{equation*}
    \begin{cases}
      -\m{L} w > 0 & \text{in } \Omega_+,\\
      \m{N}_{\var_+} w \geq 0 & \text{on } \partial \Omega_+.
    \end{cases}
  \end{equation*}
  Because $(\Omega_+, \var_+)$ is $f'(0)$-narrow, we have $\lambda(\Omega_+, \var_+) > f'(0)$.
  Since $q \leq f'(0)$, this implies that $\lambda(-\m{L}, \Omega_+, \var_+) > 0$.
  By Proposition~\ref{prop:MP}, $w \geq 0$ in $\Omega_+$, and hence in the whole domain $\Omega$.

  For each $n \in \N$, there exists $y_n \in \Omega_{-}$ such that
  \begin{equation}
    \label{eq:negative-finish}
    \left(\ubar{\mu} - \frac{1}{n}\right) u_2(y_n) - u_1(y_n) < 0.
  \end{equation}
  We take limits along $(y_n)_n$ and suppress subsequential notation.
  There exist connected limits $\Omega_-^\star \subset \Omega^\star$ and limiting solutions $u_i^\star$ for $i\in\{1, 2\}$ of \eqref{eq:main} on $(\Omega^\star, \const)$.
  Then $w^\star \coloneqq \ubar{\mu} u_2^\star - u_1^\star$ is a limit of $w$.
  It follows that $w^\star \geq 0$.
  As for $u_2^\star$ above, $f'(0)$-ampleness and Lemma~\ref{lem:zero-limit} imply that $u_i^\star > 0$ in $\Omega_-^\star$ and hence in $\Omega^\star.$
  By the strong KPP property \ref{hyp:strong} and the assumption $\ubar{\mu} > 1$, $w^\star$ satisfies
  \begin{equation*}
    -\Delta w^\star + q^\star w^\star > 0
  \end{equation*}
  for a bounded difference quotient $q^\star$.
  Thus by the strong maximum principle, $w^\star > 0$ in $\Omega^\star$.
  However, \eqref{eq:negative-finish} implies that $w^\star(0) = 0$, so $0 \in \partial \Omega^\star$.
  This contradicts the Hopf lemma unless $\const = 0$, so temporarily assume this.

  Again, we claim that $\partial_{\nu^\star} w^\star(0) = 0$, which contradicts the Hopf lemma.
  Supposing otherwise, we have ${m^\star \coloneqq -\partial_{\nu^\star} w^\star(0) > 0}$.
  We follow the reasoning from above.
  Uniform $\m{C}^{1,\al}$-regularity implies that $-\partial_\nu w \geq m^\star/2$ in a uniform neighborhood of $y_n$.
  However, $\partial_\nu u_2$ is bounded, so ${-\partial_\nu(w - u_2/n) \geq 0}$ near $y_n$ once $n$ is sufficiently large.
  Integrating from the boundary, we see that $w(y_n) - u_2(y_n)/n > 0$ for such $n$, contradicting \eqref{eq:negative-finish}.

  Returning to the general case $\const \in [0, 1]$, these contradictions imply that $\ubar{\mu} \leq 1$, so $u_1 \leq u_2$ in $\Omega_-$.
  Mimicking the maximum principle argument on $\Omega_+$, we can extend this relation to $\Omega_+,$ and thus to the entire domain $\Omega$.
  By symmetry, $u_1 = u_2$.
\end{proof}
\begin{figure}[t]
  \centering
  \includegraphics[width=0.8\textwidth]{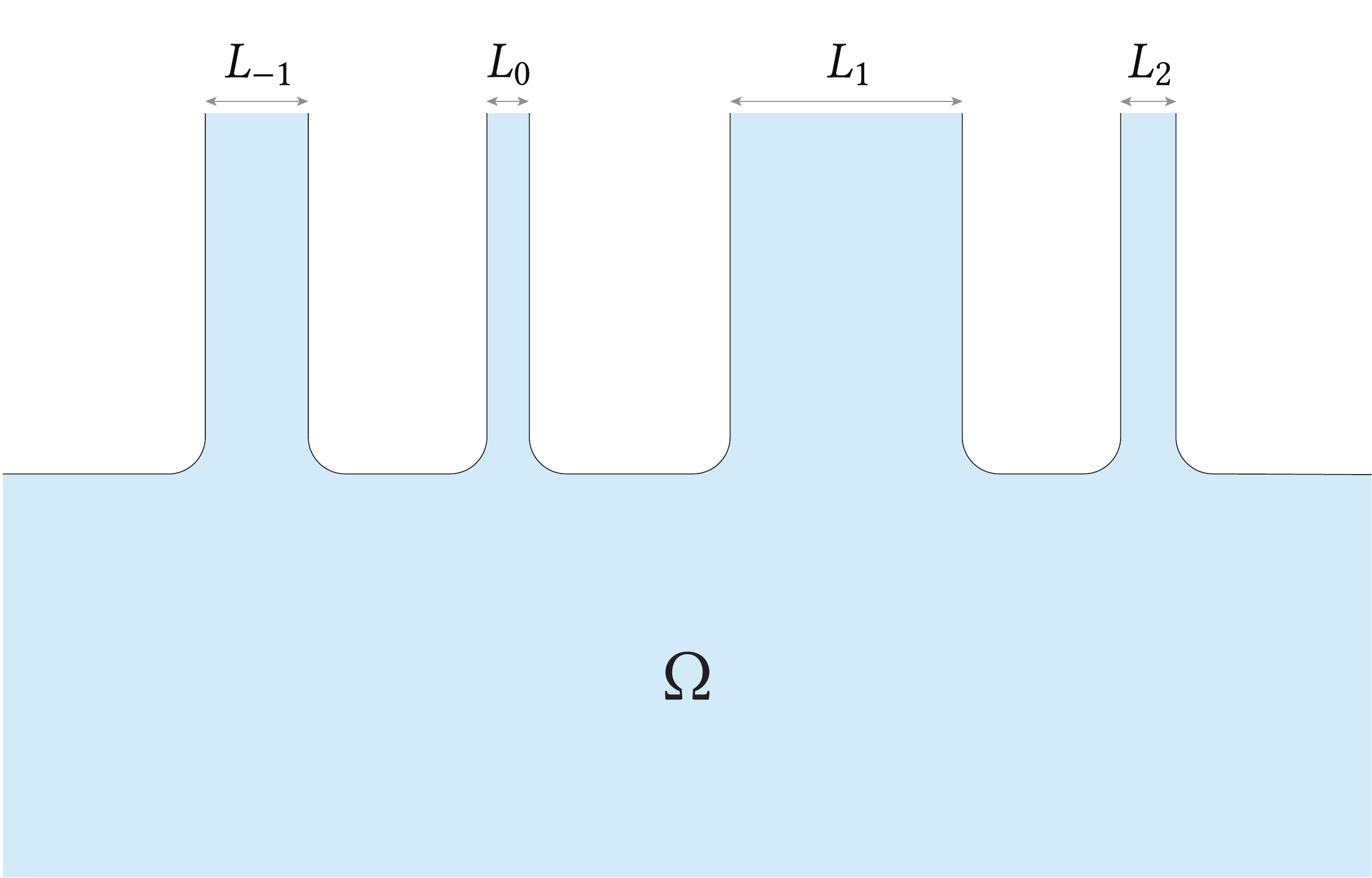}
  \caption{A comb domain in $\R^2$ whose teeth have width $L_i$ for $i \in \Z$.}
  \label{fig:noncritical}
\end{figure}
To understand when Theorem~\ref{thm:uniqueness} applies, it may help to consider the comb domain $\Omega$ in Figure~\ref{fig:noncritical}.
Because $\Omega$ contains the lower half plane (and thus arbitrarily large balls), we have $\lambda(\Omega, 0) = 0$.
If we take a sequence of points ascending ``tooth~$i$'', the corresponding connected limit is a strip of width $L_i$.
Moreover, if $L$ is a subsequential limit of $(L_i)_{i \in \Z}$, then we can take an ascending sequence of points from that subsequence of teeth to obtain a strip of width $L$ as a connected limit.
The principal eigenvalue of a strip of width $L$ is $\pi^2 L^{-2}$.
Thus, these observations imply that
\begin{equation*}
  \Sigma(\Omega, 0) = \{0\} \cup \overline{\{\pi^2L_i^{-2}\}}_{i \in \Z}.
\end{equation*}
We note that, in this case, $\Sigma(\Omega, 0)$ is closed.
Therefore Theorem~\ref{thm:uniqueness} applies to $\Omega$ if and only if $f'(0) \not\in \Sigma(\Omega, 0)$.
Equivalently, if and only if
\begin{equation*}
  \inf_{i \in \Z} \abs{f'(0) - \pi^2 L_i^{-2}} > 0.
\end{equation*}
In other words, to apply Theorem~\ref{thm:uniqueness}, $\Omega$ must be \emph{noncritical} in the sense that the widths $L_i$ uniformly avoid the critical value $\pi f'(0)^{-1/2}$.
\begin{remark}
  It seems plausible that the limit spectrum $\Sigma(\Omega, \const)$ is \emph{always} closed, but we are not presently able to show this.
\end{remark}
We close this section with the complete classification of positive bounded solutions in periodic domains.
\begin{proof}[Proof of Corollary~\textup{\ref{cor:periodic}}]
  By Theorem~\ref{thm:existence}, \eqref{eq:main} admits a positive bounded solution if and only if ${\lambda(\Omega, \const) < f'(0)}$.
  We assume this condition.
  
  Since $\Omega$ is periodic, every connected limit of $\Omega$ is a translate of $\Omega$ itself.
  It follows that $\lambda(\Omega, \const)$ is the unique limit eigenvalue.
  Hence the critical value $f'(0)$ is not in $\bar\Sigma(\Omega, \const) = \{\lambda(\Omega, \const)\}$.
  By Theorem~\ref{thm:uniqueness}, the positive bounded solution is unique.
\end{proof}

\subsection{Weak-KPP nonuniqueness}

We next demonstrate the importance of the strong KPP hypothesis \ref{hyp:strong} by constructing examples of nonuniqueness for weak-KPP reactions.
\begin{proof}[Proof of Proposition~\textup{\ref{prop:nonuniqueness}}]
  Let $\Omega$ be a bounded domain with boundary parameter ${\const \in [0, 1)}$.
  For a graphical representation of the various reactions used in this proof, see Figure~\ref{fig:reaction}.
  \begin{figure}[t]
    \centering
    \includegraphics[width=\linewidth]{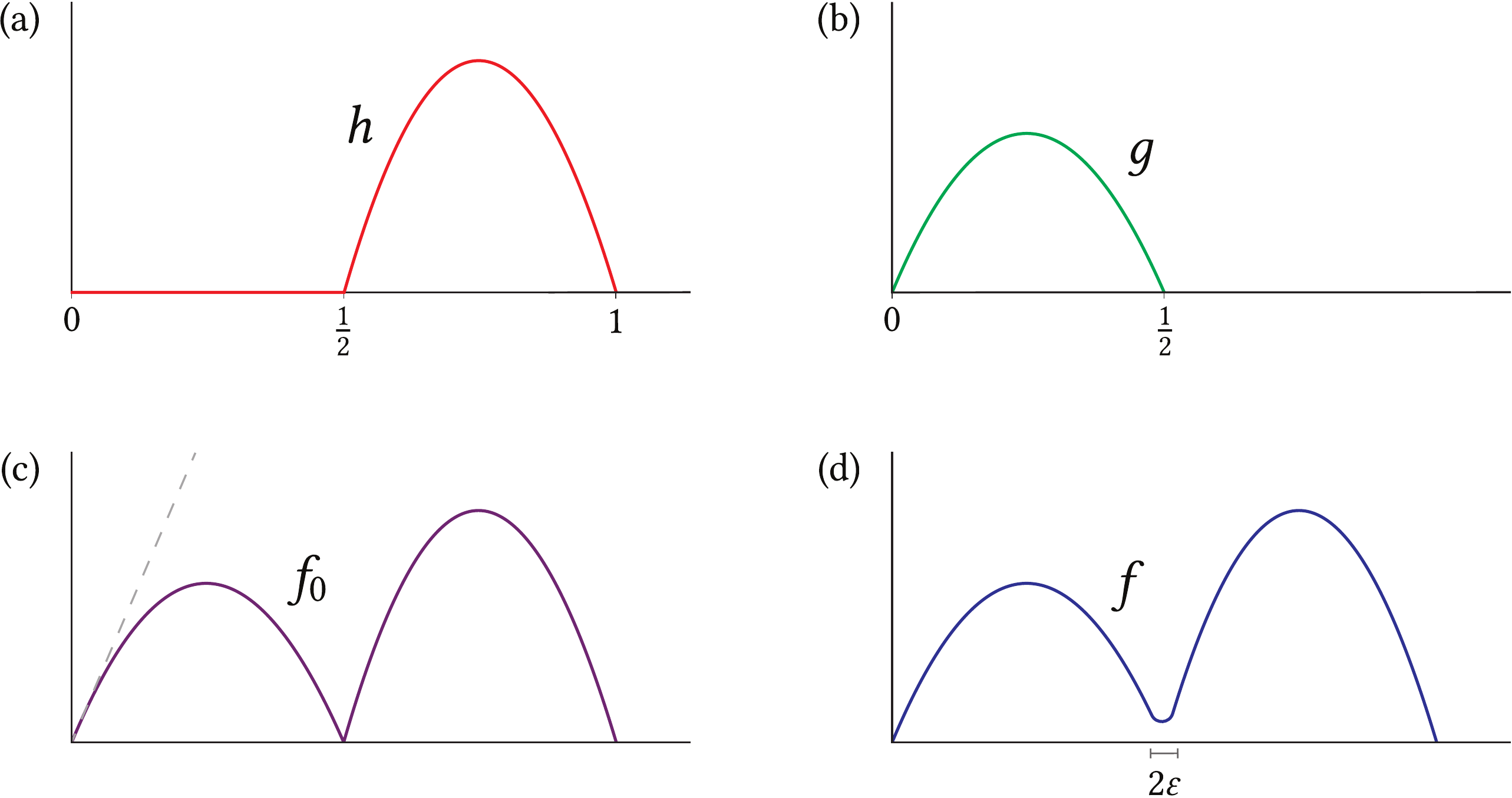}
    \caption[Reactions in the proof of nonuniqueness.]
    {
      Reactions in the proof of Proposition~\ref{prop:nonuniqueness}.
      (a) An ignition reaction.
      (b) A weak-KPP reaction on a restricted interval.
      (c) A nonnegative reaction with multiple solutions satisfying the weak-KPP inequality \ref{hyp:weak}.
      (d) A true weak-KPP reaction with multiple solutions.
    }
    \label{fig:reaction}
  \end{figure}

  Let $h \colon \R \to [0, \infty)$ satisfy $h|_{(1/2,1)^c} \equiv 0$, $h|_{(1/2,1)} > 0$, and $h \in \m{C}^{1,\gamma}([1/2, 1])$.
  Then $h$ is a so-called ``ignition reaction'' with ignition temperature $1/2$.
  We define its antiderivative
  \begin{equation*}
    H(s) \coloneqq \int_0^s h(r) \d r.
  \end{equation*}

  We claim that \eqref{eq:main} admits a positive bounded solution with reaction $h$ provided $h$ is sufficiently large.
  To see this, note that solutions of \eqref{eq:main} with reaction $h$ are critical points of the energy functional
  \begin{equation*}
    E[v] \coloneqq \int_\Omega \left[\frac{1}{2} \abs{\nab v}^2 - H(v)\right] + \const^{-1}(1 - \const)\int_{\partial \Omega} v^2,
  \end{equation*}
  where we interpret the boundary term as $0$ if $\const = 0$ and $v = 0$ on $\partial \Omega$.
  By making $h$ large near $s = 3/4$, say, we can evidently arrange that $\inf E < 0$.
  Then the direct method from the calculus of variations yields a minimizer $u_+^h$ such that $0 < u_+^h < 1$.
  This is a positive bounded solution of \eqref{eq:main}.
  We claim that $\sup u_+^h > 1/2$.
  Otherwise, $h(u_+^h) = 0$ and $u_+^h$ is harmonic.
  It thus achieves its maximum on $\partial \Omega,$ which contradicts $\m{N}_\const u_+^h = 0$ because $\const < 1$.
  So indeed $\sup u_+^h > 1/2$.

  Next, let $g$ be a weak-KPP reaction on the restricted interval $(0, 1/2)$ such that
  \begin{equation}
    \label{eq:slope}
    g'(0) > \max\left\{\lambda(\Omega, \const),\, \sup_{s \in [1/2,1]} \frac{h(s)}{s}\right\}.
  \end{equation}
  By Theorem~\ref{thm:existence}, \eqref{eq:main} has a positive bounded solution $u_-$ with reaction $g$.
  Because $g$ vanishes on $[1/2, \infty)$, we have $0 < u_- \leq 1/2$.
  Moreover, the boundary condition $\m{N}_\const u_- = 0$ ensures that $u_-$ is not constant, so by the strong maximum principle and the boundedness of $\Omega$, $0 < u_- \leq \sup u_- < 1/2$.
  We define
  \begin{equation}
    \label{eq:short}
    \eps \coloneqq 1/2 - \sup u_- > 0.
  \end{equation}

  Now let $f_0 \coloneqq g + h$.
  By \eqref{eq:slope}, $f_0$ satisfies the weak-KPP inequality \ref{hyp:weak}.
  Because $h \equiv 0$ on $[0, 1/2]$, $f_0(u_-) = g(u_-)$ and $u_-$ solves \eqref{eq:main} with reaction $f_0$.
  While the same cannot be said of $u_+^h$, we do have $f_0(u_+^h) \geq h(u_+^h)$.
  It follows that $u_+^h$ is a \emph{sub}solution for the reaction $f_0$.
  Using the parabolic semigroup $\m{P}_t^{f_0}$ defined by \eqref{eq:parabolic} with reaction $f_0$, we can construct a solution $u_+^{f_0} \coloneqq \m{P}_\infty^{f_0} u_+^h \geq u_+^h$ of \eqref{eq:main} with reaction $f_0$.
  It follows that $\sup u_+^{f_0} > 1/2$.
  Since $\sup u_- < 1/2$, we have in particular that $u_- \neq u_+^{f_0}$.
  That is, $f_0$ supports two distinct positive bounded solutions.

  The reaction $f_0$ has one deficiency: it vanishes at the intermediate point $s = 1/2$ and is merely Lipschitz there.
  Recall the ``gap'' $\eps$ from \eqref{eq:short}.
  We can increase $f_0$ on the interval $(1/2-\eps, 1/2+\eps)$ so that it becomes $\m{C}^{1, \gamma}$ and strictly positive on $(0, 1)$.
  By making a sufficiently small change, we can also ensure that the modified reaction lies under the tangent line $f_0'(0)s$.
  Let $f$ denote this modification of $f_0$.
  By construction, $f$ is a true weak-KPP reaction.

  Because $f = f_0 = g$ on $[0, 1/2 - \eps]$, $u_-$ is a solution of \eqref{eq:main} with reaction $f$.
  Moreover, because $f \geq f_0$, $u_+^{f_0}$ is a subsolution for the reaction $f$.
  Again, there exists a solution $u_+$ of \eqref{eq:main} with reaction $f$ such that $u_+ \geq u_+^{f_0}$.
  Then $\sup u_+ > 1/2$, so $u_+ \neq u_-$.
  Therefore $f$ is a weak-KPP reaction such that \eqref{eq:main} admits two distinct positive bounded solutions on $(\Omega, \const)$.
\end{proof}
In our definition of $f_0 = g + h$, $g$ only needs to be a sufficiently large KPP reaction.
In particular, if we fix $g$ and define $f_\mu \coloneqq f_0 + \mu g$ for $\mu \geq 0$, then $f_\mu$ would also support multiple positive bounded solutions on $\Omega$.
This corresponds to a taller first ``hump'' in Figure~\ref{fig:reaction}(c).

It is interesting to consider the behavior of solutions under the limit $\mu \to \infty$.
Let $u_-^\mu$ and $u_+^\mu$ denote the minimal and maximal solutions of \eqref{eq:main} on $\Omega$ with reaction $f_\mu$.
As $f_\mu$ tends to $+\infty$ on $(0, 1/2)$, it will push solutions into the range $[1/2, 1]$ in the interior of $\Omega$.
Indeed, it is straightforward to check that $u_-^\mu \to 1/2$ locally uniformly in $\Omega$ as $\mu \to \infty$.
Thus $u_-^\mu$ develops a boundary layer around $\partial \Omega$ wherein it rapidly transition from $0$ to $1/2$.

It is reasonable to expect $u_+^\mu$ to develop a similar boundary layer.
This suggests that $u_+^\mu$ converges locally uniformly to a solution of the problem
\begin{equation*}
  \begin{cases}
    -\Delta u = h(u) & \text{in } \Omega,\\
    u = 1/2 & \text{on } \partial \Omega,\\
    u > 1/2 & \text{in } \Omega.
  \end{cases}
\end{equation*}
If $h$ is strong-KPP on the interval $[1/2, 1]$, then this solution is unique.
Finally, we observe that Morse theory suggests the existence of at least one unstable solution between $u_-^\mu$ and $u_+^\mu$.
We conjecture that all such solutions converge to $1/2$ locally uniformly as $\mu \to \infty$ if $h$ is strong-KPP on $[1/2, 1]$.

\subsection{The hair-trigger effect}
To close the section, we prove the hair-trigger effect for strong-KPP reactions.
\begin{proof}[Proof of Proposition~\textup{\ref{prop:hair-trigger}}]
  First assume $f$ is positive and $\lambda(\Omega, \const) < f'(0)$.
  By Theorem~\ref{thm:existence}, \eqref{eq:main} admits a positive bounded solution.
  Because $\lambda(\Omega, \const) < f'(0)$, Proposition~\ref{prop:lambda}\ref{item:exhaustion} implies that ${\lambda(\Omega \mid B_R, \const) < f'(0)}$ for some large $R > 0.$
  Because $f$ is $\m{C}^{1, \gamma}$, there exists $\eta \in (0, 1)$ such that
  \begin{equation*}
    \inf_{s \in (0, \eta]} \frac{f(s)}{s} \geq \lambda(\Omega \mid B_R, \const).
  \end{equation*}
  Let $\varphi$ denote the positive principal eigenfunction associated to $\lambda(\Omega \mid B_R, \const)$ normalized by $\norm{\varphi}_\infty = 1$, which we extend by $0$ to all of $\Omega$.
  Then $\eps \varphi$ is a subsolution of \eqref{eq:main} for all $\eps \in [0, \eta]$.

  Let $u$ be a positive bounded solution of \eqref{eq:main}, so $u \geq \eps \varphi$ for some $\eps \in (0, \eta]$.
  Let $\bar{\eps}$ be the largest such $\eps$, and suppose for the sake of contradiction that $\bar \eps \in (0, \eta)$.
  Then $u \geq \bar \eps \varphi$, and by the strong maximum principle $u > \bar \eps \varphi$ in $\Omega$.
  Then the compactness of $\bar B_R$ and the Hopf lemma imply that $u \geq (\bar \eps + \delta) \varphi$ for some $\delta > 0$, contradicting the definition of $\bar\eps$.
  From this, we see that $u \geq \eta \varphi$.
  
  Recall the parabolic semigroup $(\m{P}_t)_{t \geq 0}$ associated with \eqref{eq:parabolic}.
  Because $\eta \varphi$ is a subsolution of \eqref{eq:main}, $\m{P}_t(\eta \varphi)$ is increasing in $t$ and thus has a limit $u_- \coloneqq \m{P}_\infty(\eta \varphi)$, which is a positive bounded solution of \eqref{eq:main}.
  By the parabolic comparison principle, $u \geq u_-$.
  That is, $u_-$ is minimal.
  Similarly, every positive bounded solution is bounded by the supersolution $1$, so $u_+ \coloneqq \m{P}_\infty 1$ is the maximal solution of \eqref{eq:main}.
  Every positive bounded solution $u$ of \eqref{eq:main} satisfies $u_- \leq u \leq u_+$.

  Now let $v$ solve \eqref{eq:parabolic} with $v(0, \anon)$ bounded, nonnegative, and not identically $0$.
  By the parabolic strong maximum principle and the Hopf lemma, $v(1, \anon) \geq \eps \varphi$ for some $\eps \in (0, \eta]$.
  Because $\eps \varphi$ is a (nonzero) subsolution of \eqref{eq:main}, $\m{P}_\infty(\eps \varphi)$ is a positive bounded solution of \eqref{eq:main}.
  Therefore $\m{P}_\infty(\eps \varphi) \geq u_-$.
  In fact, the comparison principle implies that $\m{P}_\infty(\eps \varphi) = u_-$ for all $\eps \in (0, \eta]$.
  Also, comparison yields $v(t, \anon) \geq \m{P}_{t-1}(\eps \varphi)$ for all $t \geq 1$.
  Taking $t \to \infty$, we see that
  \begin{equation}
    \label{eq:liminf}
    \liminf_{t \to \infty} v(t, \anon) \geq u_-.
  \end{equation}
  Because $M \coloneqq \sup v(0, \anon) < \infty$, similar reasoning implies that
  \begin{equation}
    \label{eq:limsup}
    \limsup_{t \to \infty} v(t, \anon) \leq \m{P}_\infty M = \m{P}_\infty 1 = u_+.
  \end{equation}

  Now suppose in addition that $f$ is strong-KPP and ${f'(0) \not\in \bar\Sigma(\Omega, \const)}$.
  By Theorem~\ref{thm:uniqueness}, \eqref{eq:main} admits a unique solution $u$.
  It follows that $u_- = u_+ = u$.
  Combining \eqref{eq:liminf} and \eqref{eq:limsup}, we obtain
  \begin{equation*}
    \lim_{t \to \infty} v(t, \anon) = u.
  \end{equation*}
  Standard parabolic estimates imply that this limit holds locally uniformly in $\bar \Omega$.
\end{proof}

\section{A critical example}
\label{sec:bulb}
Recall Conjecture~\ref{conj:KPP}: we expect uniqueness to hold in general for strong-KPP reactions.
We would like to adapt the proof of Theorem~\ref{thm:uniqueness} to general domains, but critical limit domains pose difficulties.
To probe this issue, consider the ``bulb domain'' in Figure~\ref{fig:bulb}.
\begin{figure}[ht]
  \centering
  \includegraphics[width = 0.9\linewidth]{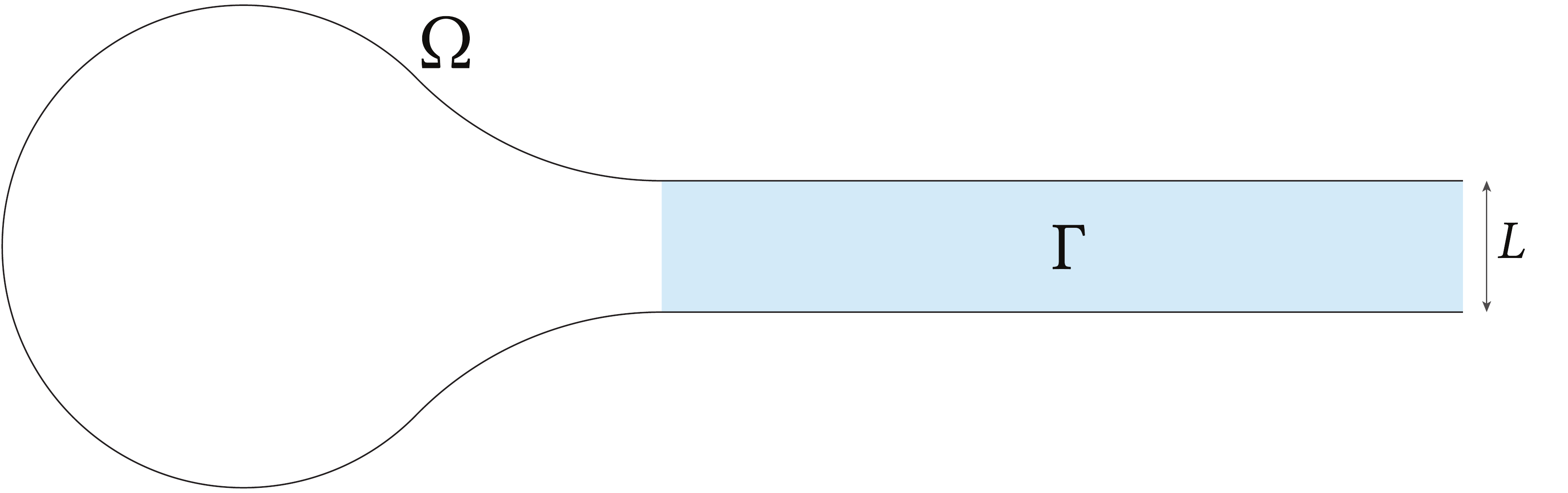}
  \caption{A bulb domain in $\R^2$.}
  \label{fig:bulb}
\end{figure}

\noindent
We impose Dirichlet conditions ($\const = 0$) and assume that the round portion is sufficiently large that $\lambda(\Omega, 0) < f'(0)$.
Then by Theorem~\ref{thm:existence}, \eqref{eq:main} admits at least one positive bounded solution on $\Omega$.
We choose coordinates so that the shaded region $\Gamma$ is $\R_+ \times (0, L)$.

We can explicitly compute the limit spectrum of $\Omega$.
Let $\Omega^*$ be a connected limit of $\Omega$ along a sequence $(x_n, y_n)_{n\in \N}$.
If $(x_n)_n$ is bounded, then $\Omega^*$ is a translate of $\Omega$ and $\lambda(\Omega^*, 0) = \lambda(\Omega, 0)$.
On the other hand, if $(x_n)_n$ tends to infinity, then $\Omega^*$ is a cylinder of width $L$, and one can easily use \eqref{eq:lambda} and Proposition~\ref{prop:lambda}\ref{item:variation} to check that
\begin{equation*}
  \lambda(\Omega^*, 0) = \lambda\big((0, L), 0\big) = \frac{\pi^2}{L^2}.
\end{equation*}
Therefore
\begin{equation*}
  \Sigma(\Omega, 0) = \big\{\lambda(\Omega, 0), \, \pi^2/L^2\big\}.
\end{equation*}
From this and our assumption that $\lambda(\Omega, 0) < f'(0)$, we see that Theorem~\ref{thm:uniqueness} applies if and only if $L \neq \pi f'(0)^{-1/2}$.

In this section, we assume that $L = \pi f'(0)^{-1/2}$ is critical.
In this case, the nonlinear structure of $f$ begins to play a role.
We assume that $f$ has the form
\begin{equation}
  \label{eq:nonlinear}
  f(s) = f'(0)s - Bs^2 + \m{O}(s^3)
\end{equation}
for some $B > 0$.
\begin{lemma}
  \label{lem:critical-bulb}
  Let $f$ be a strong-KPP reaction satisfying \eqref{eq:nonlinear}.
  Then every solution $u$ of \eqref{eq:main} on $(\Omega, 0)$ satisfies
  \begin{equation}
    \label{eq:polynomial-decay}
    u(x, y) \sim \frac{9\pi}{4B x^2} \sin\left(\frac{\pi y}{L}\right)
  \end{equation}
  uniformly in $y$ as $x \to \infty$.
\end{lemma}
This resolves one of the main questions in our proof of uniqueness---it shows that any two solutions are comparable.
Our proof will make use of the following characterization of the principal eigenvalue problem on cylinders.
\begin{lemma}
  \label{lem:symmetry}
  Fix $\const \in [0, 1]$ and suppose $\Omega = \R \times \omega$ for some bounded domain $\omega \subset \R^{d - 1}$.
  Then $\lambda(\Omega, \const) = \lambda(\omega, \const)$.
  Moreover, if $\varphi$ is a positive principal eigenfunction on $(\Omega, \const)$ in the sense that
  \begin{equation*}
    \begin{cases}
      -\Delta \varphi = \lambda(\Omega, \const) \varphi & \text{in } \Omega,\\
      \m{N}_\const \varphi = 0,
    \end{cases}
  \end{equation*}
  then $\varphi$ is a multiple of the principal eigenfunction on $(\omega, \const)$.
  In particular, $\varphi$ is independent of the first coordinate.
\end{lemma}
We defer the proof to the end of the section.
\begin{proof}[Proof of Lemma~\textup{\ref{lem:critical-bulb}}]
  Let $u$ be a positive bounded solution of \eqref{eq:main} and consider its translates $\tau_{(n,0)}u$.
  As $n \to \infty$, we can extract a subsequence that converges to a bounded nonnegative limit $u^*$ solving \eqref{eq:main} on the cylinder $\R \times (0, L)$.
  We know $u^*$ lies under the supersolution $1$.
  Recall the parabolic evolution from \eqref{eq:parabolic}.
  By the parabolic comparison principle, $0 \leq u^* \leq \m{P}_\infty 1$.
  But $\m{P}_\infty 1$ is independent of $x$, and thus corresponds to a nonnegative bounded solution of \eqref{eq:main} on the cross-section $(0, L)$.
  This problem is critical, so by Theorem~\ref{thm:bounded} for bounded domains, $\m{P}_\infty 1 = 0$.
  Hence $u^* = 0$.
  The limit is unique, so in fact $u$ converges to $0$ uniformly as $x \to \infty$.
  We determine the asymptotics of this limit.

  Let $\{\varphi_k\}_{k \in \N}$ denote an orthonormal basis of eigenfunctions of $-\Delta$ on $(0, L)$, so that
  \begin{equation}
    \label{eq:Fourier-basis}
    \varphi_k(y) = \sqrt{\frac{2}{L}} \sin \left(\frac{\pi k y}{L}\right)
  \end{equation}
  for $k \in \N$.
  These have eigenvalues $(\pi k/L)^2.$
  Since $L = \pi f'(0)^{-1/2}$, the principal eigenvalue is $f'(0)$ and $\varphi_1$ solves the linearization of \eqref{eq:main} about $0$ on $\R \times (0, L)$.
  However, because $f'(0)u - f(u) > 0$ for $u > 0$, there is nonlinear absorption.
  We have assumed that it takes the standard form \eqref{eq:nonlinear}.

  Because $u \in \m{C}_{\mathrm{b}}^{2, \gamma}$, its Fourier series
  \begin{equation}
    \label{eq:Fourier-series}
    u(x, y) = \sum_{k \geq 1} \al_k(x) \varphi_k(y)
  \end{equation}
  converges absolutely and $\al_k \in \m{C}_{\mathrm{b}}^{2,\gamma}(\R_+)$.
  Let $\braket{\anon, \anon}$ denote the inner product on $L^2\big((0, L)\big)$.
  Since $u(x, \anon) \to 0$ uniformly as $x \to \infty$, the inner product representation
  \begin{equation}
    \label{eq:inner-product}
    \al_k(x) = \braket{u(x, \anon), \varphi_k}
  \end{equation}
  implies that $\al_k(x) \to 0$ as $x \to \infty$.

  We claim that the first term of \eqref{eq:Fourier-series} dominates as $x \to \infty$.
  To see this, consider the sequence $v_n \coloneqq \tau_{(n, 0)}u/\al_1(n)$.
  By \eqref{eq:inner-product}, we have
  \begin{equation}
    \label{eq:transverse}
    \braket{v_n(0, \anon), \varphi_1} = 1.
  \end{equation}
  Since $\al_1(n) \to 0$ as $n \to \infty$, $v_n$ converges along a subsequence to a solution $v^*$ of the linearization of \eqref{eq:main} about $0$ on $\R \times (0, L)$.
  By Lemma~\ref{lem:symmetry}, $v^*$ is a multiple of the transverse principal eigenfunction $\varphi_1$.
  Moreover, \eqref{eq:transverse} yields $\braket{v^*, \varphi_1} = 1,$ so $v^* = \varphi_1$.
  Since the limit is unique, $v_n$ converges locally uniformly along the whole sequence to $\varphi_1$.
  That is,
  \begin{equation}
    \label{eq:principal}
    \lim_{x \to \infty} \frac{u(x, y)}{\al_1(x)} = \varphi_1(y).
  \end{equation}
  Elliptic estimates imply that this limit is uniform in $y$ and valid in $\m{C}^2$.
  Taking the inner product between \eqref{eq:principal} and $\varphi_k$, we see that $\al_k \ll \al_1$ as $x \to \infty$ for all $k \geq 2$.

  Write $u = \al_1\varphi_1 + u^\perp$.
  Substituting this in \eqref{eq:main} and using \eqref{eq:nonlinear}, we obtain
  \begin{equation*}
    -\al_1''\varphi_1 - \Delta u^\perp = f'(0)u^\perp - B\varphi_1^2 \al_1^2 - 2 B \varphi_1 \al_1 u^\perp - B \big(u^\perp\big)^2 + \m{O}(u^3).
  \end{equation*}
  We take the inner product with $\varphi_1$ and observe that $\braket{\Delta u^\perp, \varphi_1} = 0 = \Braket{u^\perp, \varphi_1}$.
  Also, \eqref{eq:principal} yields $u^\perp \ll \al_1 \varphi_1$ and $u^3 \ll \varphi_1^2 \al_1^2$.
  Therefore
  \begin{equation}
    \label{eq:approx-ODE}
    \al_1'' = B \braket{\varphi_1^2, \varphi_1} \al_1^2 \left[1 + \beta(x)\right]
  \end{equation}
  for some $\beta(x) \to 0$ as $x \to \infty$.
  Define
  \begin{equation*}
     \zeta \coloneqq B \braket{\varphi_1^2, \varphi_1}, \quad \ubar{\beta}(x) \coloneqq \inf_{[x, \infty)} \beta, \And \bar{\beta}(x) \coloneqq \sup_{[x, \infty)} \beta.
  \end{equation*}
  Then $\ubar\beta$ and $\bar\beta$ are monotone increasing and decreasing towards $0$, respectively.
  Fix $m$ such that $\ubar{\beta}(m) \geq -1$.
  Then on $[m, \infty)$, \eqref{eq:approx-ODE} implies that $\al_1$ is convex and decreasing towards $0$.
  We assume for the time being that $x \geq m$.
  We note that elliptic estimates imply that $\al_1' \to 0$ as $x \to \infty$ as well.
  Multiplying \eqref{eq:approx-ODE} by $\al_1'$ and integrating on $(x, \infty)$, we find
  \begin{equation*}
    \frac{2\zeta}{3}\al_1^3(1 + \ubar \beta) \leq (\al_1')^2 \leq \frac{2\zeta}{3}\al_1^3(1 + \bar \beta).
  \end{equation*}
  Rearranging, we have
  \begin{equation*}
    \sqrt{\frac{\zeta}{6}(1 + \ubar\beta)} \leq \left(\frac{1}{\sqrt{\al_1}}\right)' \leq \sqrt{\frac{\zeta}{6}(1 + \bar\beta)}.
  \end{equation*}
  Integrating on $(x_0, x)$ and taking $x \to \infty$, we can check that
  \begin{equation*}
    \left[\frac{\zeta}{6}\big(1 + \bar\beta(x_0)\big)\right]^{-1} \leq \liminf_{x \to \infty} x^2\al_1(x) \leq \limsup_{x \to \infty} x^2\al_1(x) \leq \left[\frac{\zeta}{6}\big(1 + \ubar\beta(x_0)\big)\right]^{-1}.
  \end{equation*}
  Since $\ubar\beta,\bar\beta \to 0$ as $x \to \infty$, we can take $x_0 \to \infty$ to see that
  \begin{equation*}
    \lim_{x \to \infty} x^2 \al_1(x) = \frac{6}{\zeta}.
  \end{equation*}
  That is,
  \begin{equation}
    \label{eq:leading}
    \al_1(x) \sim \frac{6}{B \braket{\varphi_1^2, \varphi_1} x^2} \quad \text{as } x \to \infty.
  \end{equation}
  We can use \eqref{eq:Fourier-basis} to compute
  \begin{equation*}
    \frac{\varphi_1(y)}{\braket{\varphi_1^2, \varphi_1}} = \frac{3\pi}{8} \sin\left(\frac{\pi y}{L}\right).
  \end{equation*}
  Then \eqref{eq:polynomial-decay} follows from \eqref{eq:principal} and \eqref{eq:leading}.
  Elliptic estimates up to the boundary imply that \eqref{eq:polynomial-decay} holds uniformly in $y$.
\end{proof}
With explicit knowledge of the behavior of solutions at infinity, we can easily show uniqueness.
\begin{proposition}
  Let $f$ be a strong-KPP reaction satisfying \eqref{eq:nonlinear}.
  Then \eqref{eq:main} has a unique positive bounded solution on $(\Omega, 0)$.
\end{proposition}
\begin{proof}
  Suppose $u_1$ and $u_2$ are both positive bounded solutions of \eqref{eq:main}.
  Following the proof of Theorem~\ref{thm:uniqueness}, define
  \begin{equation*}
    \ubar{\mu} \coloneqq \inf\{\mu \mid u_1 \leq \mu u_2\}.
  \end{equation*}
  By the strong maximum principle and the Hopf lemma, $u_1$ and $u_2$ are comparable on any compact $K \Subset \bar\Omega$.
  That is, there exists $C(K) \geq 1$ such that
  \begin{equation*}
    C(K)^{-1} u_2 \leq u_1 \leq C(K) u_2 \quad \text{in } K.
  \end{equation*}
  On the other hand, Lemma~\ref{lem:critical-bulb} implies that $u_1 \sim u_2$ at infinity.
  It follows that $\ubar{\mu} < \infty$.

  Suppose for the sake of contradiction that $\ubar{\mu} > 1$.
  Define
  \begin{equation*}
    w \coloneqq \ubar{\mu} u_2 - u_1 \geq 0.
  \end{equation*}
  As in the proof of Theorem~\ref{thm:uniqueness}, the strong KPP property \ref{hyp:strong} ensures that $-\Delta w > qw$ for a suitable difference quotient $q \in L^\infty(\Omega)$.
  By the strong maximum principle, $w > 0$.

  When $x \to \infty$, Lemma~\ref{lem:critical-bulb} implies that $w \sim (\ubar{\mu} - 1)u_2$ uniformly in $y$.
  Thus if $\delta \in (0, \ubar{\mu} - 1)$, there exists compact $K \Subset \Omega$ such that $w \geq \delta u_2$ on $\Omega \setminus K$.
  Moreover, the strong maximum principle and the Hopf lemma imply that $w$ and $u_2$ are comparable on $K$.
  In particular, perhaps after reducing $\delta > 0$, we have $w \geq \delta u_2$ on $K$.
  Hence $w \geq \delta u_2$ on the entire domain $\Omega$.
  This contradicts the definition of $\ubar{\mu}$, so in fact $\ubar{\mu} \leq 1$.
  Uniqueness follows by symmetry.
\end{proof}
We close the section with a proof of our symmetry result.
\begin{proof}[Proof of Lemma~\textup{\ref{lem:symmetry}}]
  We denote coordinates on $\Omega = \R \times \omega$ by $(x, y)$.
  Let $\psi$ be a positive principal eigenfunction on $(\omega, \const)$.
  Then we can view $\psi$ as a function on $\Omega$ as well; using it in \eqref{eq:lambda}, we see that $\lambda(\Omega, \const) \geq \lambda(\omega, \const)$.
  On the other hand, we can use the test functions $\cos(\eps x) \tbf{1}_{\abs{x} < \pi/2\eps} \psi(y)$ in \eqref{eq:variation} to verify the opposite inequality when $\eps \searrow 0$.
  Thus $\lambda(\Omega, \const) = \lambda(\omega, \const).$

  Now let $\varphi$ be a positive principal eigenfunction on $(\Omega, \const)$ and define
  \begin{equation*}
    E(x) \coloneqq \int_\omega \varphi(x, y) \psi(y) \d y.
  \end{equation*}
  Integrating by parts several times and using $\m{N}_\const \varphi = \m{N}_\const \psi = 0$, we can check that $E'' = 0$.
  Since $\varphi > 0,$ $E$ is a positive affine function, and hence constant.
  By Harnack and the Hopf lemma, $\varphi$ is comparable to $\psi$ in the sense that
  \begin{equation*}
    C^{-1} \psi \leq \varphi \leq C \psi
  \end{equation*}
  for some $C \geq 1$.
  Let
  \begin{equation*}
    \ubar{\mu} \coloneqq \inf\{\mu > 0 \mid \varphi \leq \ubar{\mu} \psi \}.
  \end{equation*}
  Then $\ubar{\mu} \in \R_+$.
  Consider $w \coloneqq \ubar{\mu} \psi - \varphi \geq 0$.
  If $w = 0$, we are done.
  Otherwise, the strong maximum principle implies that $w > 0$.
  Thus $w$ is a positive principal eigenfunction on $(\Omega, \const)$.
  Replacing $\varphi$ by $w$ in our work above, we see that $w$ is itself comparable to $\psi$.
  In particular, there exists $\eps > 0$ such that $w \geq \eps \psi$.
  Therefore
  \begin{equation*}
    \varphi \leq (\ubar{\mu} - \eps) \psi,
  \end{equation*}
  contradicting the definition of $\ubar{\mu}$.
  So in fact $\varphi = \ubar{\mu} \psi$.
\end{proof}

\appendix

\section{Smoothness and convergence of domains}
\label{sec:smooth}

In this appendix, we precisely define our notion of smoothness for domains and prove an essential compactness result.
These ideas are widely understood folklore, but we use them so extensively that we elect to put the theory on solid ground.

We begin with some preliminary notation.
Given $r > 0$, we let $B_r$ and $B_r^{d - 1}$ denote the $r$-balls in $\R^d$ and $\R^{d - 1}$, respectively.
We define the cylinder
\begin{equation*}
  \Gamma_r \coloneqq B_r^{d - 1} \times (-r, r),
\end{equation*}
on which we use coordinates $(x', y) \in \R^{d - 1} \times \R$.
Given a function $\phi \in \m{C}^{2,\gamma}(B_r^{d-1})$, we define its hypograph
\begin{equation*}
  \Pi_r(\phi) \coloneqq \big\{(x', y) \in \Gamma_r \mid y < \phi(x')\big\}.
\end{equation*}
Such hypographs serve as building blocks for our smooth sets.
Finally, we let $\op{Isom}_0(\R^d)$ denote the group of orientation-preserving isometries of $\R^d$.
It is generated by translations and rotations.
\begin{definition}
  \label{def:smooth}
  Given $r \in (0, 1]$ and $K > 0,$ we say that an open set $\Omega \subset \R^d$ is $(\gamma,r,K)$-\emph{smooth} if one of following holds for every $x \in \R^d$:
  \begin{enumerate}[label = \textup{(S\arabic*)}, itemsep = 1ex]
  \item
    \label{item:smooth-subset}
    $B_{r/8}(x) \subset \Omega$;

  \item
    \label{item:smooth-disjoint}
    $B_{r/8}(x) \subset \Omega^c$;

  \item
    \label{item:smooth-boundary}
    There exist $g \in \op{Isom}_0(\R^d)$ and $\phi \in \m{C}^{2,\gamma}\big(B_r^{d-1}; [-r/4,r/4]\big)$ such that ${g(x) = 0}$, $\norm{\phi}_{\m{C}^{2,\gamma}(B_r^{d - 1})} \leq K$, and $g \Omega \cap \Gamma_r = \Pi_r(\phi)$.
  \end{enumerate}
  It follows that $\partial \Omega$ is a uniformly $\m{C}^{2,\gamma}$ submanifold of $\R^d$.
  We say $\var \colon \partial \Omega \to [0, 1]$ is uniformly $\m{C}^{2,\gamma}$ smooth when this holds in the standard sense used in differential geometry.
  Also, we say a collection $\m{U}$ of open sets is $(\gamma,r,K)$-smooth if every element of $\m{U}$ is $(\gamma,r,K)$-smooth.
\end{definition}
\noindent
We often suppress parameters and say that $\Omega$ or $\m{U}$ is ``uniformly $\m{C}^{2,\gamma}$'' or ``uniformly smooth.''
\begin{remark}
  It may seem strange that we restrict the range of $\phi$ to $[-r/4, r/4]$ in \ref{item:smooth-boundary}.
  For instance, this appears to rule out the hypograph $\Pi(\psi)$ of a uniformly $\m{C}^{2,\gamma}$ function $\psi \colon \R^{d - 1} \to \R$ with $\op{Lip} \psi \gg 1$.
  While this is true for a fixed value of $r$, we claim that such a hypograph satisfies Definition~\ref{def:smooth} for sufficiently small $r$.

  This is due to the fact that the curvature of $\partial \Pi(\psi)$ is uniformly bounded.
  Thus when $r \ll 1$, $\partial \Pi(\psi)$ appears almost flat at scale $r$.
  By suitably orienting the isometry $g$, we can arrange that every local chart $\phi$ for $\Pi(\psi)$ in \ref{item:smooth-boundary} satisfies $\abs{\phi(0)} \leq r/8$ and $\abs{\nab \phi} \ll 1$.
  Thus if $r$ is small, we can impose the additional condition $\abs{\phi} \leq r/4$.
\end{remark}
Next, we define a notion of smooth convergence for smooth sets.
We use the notation $\N_* \coloneqq \N \cup \{\infty\}$ for the extended natural numbers.
When we write $n \geq N$ in the following definition, we include the limit index $n = \infty$.
\begin{definition}
  \label{def:limit}
  Given a $(\gamma, r, K)$-smooth sequence $(\Omega_n)_{n \in \N_*}$ and $\al \in (0, \gamma)$, we say ${\Omega_n \to \Omega_\infty}$ locally uniformly in $\m{C}^{2,\al}$ as $n \to \infty$ if one of the following holds for every $x \in \R^d$:
  \begin{enumerate}[label = \textup{(L\arabic*)}, itemsep = 1ex]
  \item
    \label{item:limit-subset}
    $B_{r/16}(x) \subset \Omega_n \cap \Omega_\infty$ when $n$ is sufficiently large;

  \item
    \label{item:limit-disjoint}
    $B_{r/16}(x) \subset \Omega_n^c \cap \Omega_\infty^c$ when $n$ is sufficiently large;

  \item
    \label{item:limit-boundary}
    There exist $N \in \N$, $g \in \op{Isom}_0(\R^d)$, and $(\phi_n)_{n \geq N} \subset \m{C}^{2,\gamma}\big(B_{r/2}^{d-1}; [-r/2,r/2]\big)$ such that $g(x) = 0$,
    ${g \Omega_n \cap \Gamma_{r/2} = \Pi_{r/2}(\phi_n)}$ for all $n \geq N,$ and $\phi_n \to \phi_\infty$ in $\m{C}^{2,\al}(B_{r/2}^{d - 1})$ as $n \to \infty$.
  \end{enumerate}
  Then one can naturally define locally uniform $\m{C}^{2,\al}$ convergence for boundary parameters $\var_n$ and functions $u_n$ defined on $\partial\Omega_n$ and $\Omega_n$, respectively.
  The details are standard given the charts in \ref{item:limit-subset}--\ref{item:limit-boundary}; we omit them.
\end{definition}
We can now state our compactness result.
\begin{proposition}
  \label{prop:compactness}
  Let $(\Omega_n)_{n \in \N}$ be a $(\gamma, r, K)$-smooth sequence of domains.
  Then for each $\al \in (0, \gamma)$, there exist $(\gamma, r, K)$-smooth $\Omega_\infty$ and a subsequence $(\Omega_{n_m})_{m \in \N}$ of $(\Omega_n)_{n \in \N}$ such that ${\Omega_{n_m} \to \Omega_\infty}$ locally uniformly in $\m{C}^{2,\al}$ as $m \to \infty$.
\end{proposition}
\noindent
This is essentially Arzel\`a--Ascoli for uniformly smooth domains.
\begin{remark}
  Definition~\ref{def:limit} seems related to the notion of Cheeger--Gromov convergence from differential geometry.
  For instance, if our domains are uniformly $\m{C}^3$, we can apply Theorem~3.1 of \cite{AKKLT} to infer locally-uniform subsequential convergence as in Proposition~\ref{prop:compactness}.
\end{remark}
\begin{corollary}
  \label{cor:compactness}
  Every uniformly smooth sequence $(\Omega_n, \var_n, u_n)_{n \in \N}$ of domains $\Omega_n$, boundary parameters $\var_n \colon \partial\Omega_n \to [0, 1]$, and functions $u_n \colon \Omega_n \to \R$ has a subsequence that converges locally uniformly to a uniformly smooth limit $(\Omega_\infty, \var_\infty, u_\infty)$.
\end{corollary}
\noindent
This follows from Proposition~\ref{prop:compactness} for $(\Omega_n)_{n \in \N}$ and Arzel\`a--Ascoli for $(\var_n, u_n)_{n \in \N}$; we omit the details.
We use Corollary~\ref{cor:compactness} incessantly throughout the paper, often without explicit reference.
\begin{proof}[Proof of Proposition~\textup{\ref{prop:compactness}}]
  Let $\big(B^{(m)}\big)_{m \in \N}$ be an open cover of $\R^d$ by balls of radius $r/8$ such that every ball of radius $r/16$ is contained in some $B^{(m)}$.
  Let $x^{(m)}$ denote the center of each ball.
  We will construct nested subsequences
  \begin{equation*}
    (\Omega_n)_n \supset \ldots \supset (\Omega_{m,n})_n \supset (\Omega_{m + 1, n})_n \supset \ldots
  \end{equation*}
  indexed by $m \in \N$.
  The subsequence $(\Omega_{m, n})_n$ will determine the limit $\Omega_\infty$ within the ball $B^{(m)}$.
  Our final subsequence $(\Omega_{n_m})_m$ will be diagonal: $(\Omega_{n,n})_n$.
  As a subsequence of each $(\Omega_{m,n})_n$, it will determine a limit $\Omega_\infty$ consistently within each ball $B^{(m)}$, and thus on the whole space $\R^d$.

  We proceed inductively.
  Given $m \in \N$, suppose we have constructed $(\Omega_{m - 1, n})_n$.
  (When $m = 1$, we adopt the convention that $\Omega_{0, n} = \Omega_n$.)
  We work by cases corresponding to those in Definition~\ref{def:limit}.
  \begin{enumerate}[label = (\roman*), leftmargin=20pt, itemsep = 1ex]
  \item
    \label{item:compactness-subset}
    If $B^{(m)} \subset \Omega_{m - 1, n}$ for infinitely many $n$, let $(\Omega_{m, n})_n$ be the corresponding subsequence and let $\Omega_\infty \cap B^{(m)} = B^{(m)}$.

  \item
    \label{item:compactness-disjoint}
    Otherwise, if $B^{(m)} \subset \Omega_{m - 1, n}^c$ for infinitely many $n$, let $(\Omega_{m, n})_n$ be the corresponding subsequence and let $\Omega_\infty \cap B^{(m)} = \emptyset$.

  \item
    \label{item:compactness-boundary}
    If neither \ref{item:compactness-subset} nor \ref{item:compactness-disjoint} applies, then $\partial\Omega_{m-1,n}$ intersects $B^{(m)} = B_{r/8}\big(x^{(m)}\big)$ for all sufficiently large $n$.
    Dropping the early entries, we may assume without loss of generality that this holds for all $n$.
    We apply Definition~\ref{def:smooth} to the $(\gamma, r, K)$-smooth sets $\Omega_{m-1,n}$ around $x^{(m)} \in \R^d$.
    Since neither \ref{item:smooth-subset} nor \ref{item:smooth-disjoint} holds, \ref{item:smooth-boundary} yields $g_n \in \op{Isom}_0(\R^d)$ and ${\phi_n \in \m{C}^{2,\gamma}\big(B_r^{d-1}; [-r/4, r/4]\big)}$ such that $g_n(x^{(m)}) = 0$, ${\norm{\phi_n}_{\m{C}^{2,\gamma}(B_r^{d - 1})} \leq K}$, and $g_n \Omega_{m-1,n} \cap \Gamma_r = \Pi_r(\phi_n)$ for all $n \in \N$.

    Now, the set of orientation-preserving isometries sending $x^{(m)}$ to $0$ is homeomorphic to $\mathrm{SO}(d)$, and thus compact.
    It follows that we can extract a subsequential limit $g_\infty$ of $(g_n)_n$.
    Likewise, Arzel\`a--Ascoli allows us to extract a further subsequence of $(\phi_n)_n$ that converges in $\m{C}^{2,\al}(B_r^{d - 1})$ to some limit ${\phi_\infty \in \m{C}^{2, \gamma}\big(B_r^{d-1}; [-r/4, r/4]\big)}$.
    We let $(\Omega_{m,n})_n$ be the corresponding subsequence of $(\Omega_{m-1,n})_n$.
    We then set
    \begin{equation*}
      \Omega_\infty \cap B^{(m)} = \big(g_\infty^{-1}\Pi_r(\phi_\infty)\big) \cap B^{(m)}.
    \end{equation*}
  \end{enumerate}
  \medskip

  We have now constructed $(\Omega_{m,n})_n$.
  Iterating, we obtain our entire sequence of subsequences.
  Because the subsequences are nested, our construction of the limit $\Omega_\infty$ is consistent across different balls $B^{(m)}$.
  It is straightforward to check that $\Omega_\infty$ is uniformly $(\gamma,r,K)$-smooth.

  It remains to show that the diagonal sequence $(\Omega_{n,n})_n$ converges to $\Omega_\infty$.
  Fix $x \in \R^d$ and suppose that neither \ref{item:limit-subset} nor \ref{item:limit-disjoint} holds when we replace the sequence $(\Omega_n)_n$ by $(\Omega_{n,n})_n$.
  By our choice of the cover $\big(B^{(m)}\big)_{m \in \N}$, the ball $B_{r/16}(x)$ is contained in some $B^{(m)}$.
  We constructed $\Omega_\infty \cap B^{(m)}$ in Stage~$m$ of the above iteration, and the failure of \ref{item:limit-subset} and \ref{item:limit-disjoint} implies that neither \ref{item:compactness-subset} and \ref{item:compactness-disjoint} held in that stage.
  Thus we used \ref{item:compactness-boundary} to construct $\Omega_\infty \cap B^{(m)}$.

  Let $g_n^{(m)}$ and $\phi_n^{(m)}$ denote the corresponding isometries and height functions.
  Set $g \coloneqq g_\infty^{(m)}$ and $\phi \coloneqq \phi_\infty^{(m)}$.
  The rotations $g \big(g_n^{(m)}\big)^{-1}$ converge to the identity as $n \to \infty$.
  Thus when $n$ is sufficiently large, the graph of $\phi_n^{(m)}$ over $B_r^{d-1}$ is rotated by $g \big(g_n^{(m)}\big)^{-1}$ to the graph of a function $\phi_n$ over $B_{r/2}^{d-1}$.
  It follows that $g\Omega_{m,n} \cap \Gamma_{r/2} = \Pi_{r/2}(\phi_n)$.

  Moreover, because $\big|\phi_n^{(m)}\big| \leq r/4$, we have $\abs{\phi_n} \leq r/2$ once $n$ is large.
  By construction, $\phi_n^{(m)} \to \phi_\infty^{(m)}$ in $\m{C}^{2,\al}(B_r^{d-1})$ as $n \to \infty$.
  It is thus easy to check that $\phi_n \to \phi$ in $\m{C}^{2,\al}(B_{r/2}^{d-1})$.
  Since $(\Omega_{n,n})_n$ is a subsequence of $(\Omega_{m, n})_n$, \ref{item:limit-boundary} holds and indeed $\Omega_{n,n} \to \Omega_\infty$ locally uniformly in $\m{C}^{2,\al}$ as $n \to \infty$.
\end{proof}

\printbibliography
\end{document}